\documentclass[10pt]{amsart}
\providecommand{\XLeftMargin}{2.5cm}
\providecommand{\XTopMargin}{2.5cm}
\providecommand{\XRightMargin}{2.5cm}
\providecommand{\XBottomMargin}{2.5cm}

\usepackage{amssymb,latexsym,amsmath,amsthm,amscd}
\usepackage[all]{xy}
\xyoption{arc}
\providecommand{\XLeftMargin}{2cm}
\providecommand{\XTopMargin}{1.8cm}
\providecommand{\XRightMargin}{2cm}
\providecommand{\XBottomMargin}{1.8cm}

\usepackage[left=\XLeftMargin,top=\XTopMargin,right=\XRightMargin,bottom=\XBottomMargin]{geometry}
\usepackage{graphicx}
\usepackage{todonotes}

\newlength\XXXMyLength\makeatletter
\@ifclassloaded{amsart}{\def\Xadjustleft[#1]{}
}{
\def\Xadjustleft[#1]{\setlength\XXXMyLength{#1}\ifnum\numexpr\leftmargin>\numexpr\XXXMyLength\hspace{-\XXXMyLength}\else\hspace{-\leftmargin}\fi}
}
\makeatother

\UseComputerModernTips

\theoremstyle{plain}
\newtheorem{thm}{Theorem}[section]
\newtheorem*{thm*}{Theorem}
\newtheorem{lemma}[thm]{Lemma}

\newtheorem*{lem*}{Lemma}
\newtheorem{prop}[thm]{Proposition}
\newtheorem*{prop*}{Proposition}
\newtheorem*{claim}{Claim}
\newtheorem{cor}[thm]{Corollary}
\newtheorem*{cor*}{Corollary}

\newtheorem*{conj*}{Conjecture}

\theoremstyle{definition}
\newtheorem{cons}[thm]{Construction}
\newtheorem*{cons*}{Construction}
\newtheorem{df}[thm]{Definition}
\newtheorem*{df*}{Definition}
\newtheorem{nota}[thm]{Notation}
\newtheorem*{nota*}{Notation}

\newtheorem*{qu*}{Question}

\newtheorem{rmk}[thm]{Remark}
\newtheorem*{rmk*}{Remark}

\newtheorem{ex}[thm]{Example}
\newtheorem*{ex*}{Example}

\newcommand{\bC}{\mathbb{C}}

\newcommand{\bH}{\mathbb{H}}

\newcommand{\bO}{\mathbb{O}}

\newcommand{\bR}{\mathbb{R}}

\newcommand{\fp}{\mathfrak p}

\DeclareMathOperator{\Ad}{Ad}

\DeclareMathOperator{\Aut}{Aut}

\DeclareMathOperator{\Cor}{Cor}

\DeclareMathOperator{\End}{End}

\DeclareMathOperator{\Gal}{Gal}

\DeclareMathOperator{\Hom}{Hom}

\DeclareMathOperator{\im}{im}
\DeclareMathOperator{\Inn}{Inn}
\DeclareMathOperator{\Int}{Int}

\DeclareMathOperator{\Res}{Res}

\DeclareMathOperator{\Tr}{Tr}

\DeclareMathOperator{\SO}{SO}
\DeclareMathOperator{\PSO}{PSO}

\DeclareMathOperator{\SU}{SU}

\DeclareMathOperator{\SP}{Sp}

\DeclareMathOperator{\Orth}{O}
\DeclareMathOperator{\Spin}{Spin}

\DeclareMathOperator{\Clif}{C}

\newcommand{\injects}{\hookrightarrow}

\newcommand{\mtx}[4]{\left(\begin{matrix}#1&#2\\#3&#4\end{matrix}\right)}
\newcommand{\smtx}[4]{\left(\begin{smallmatrix}#1&#2\\#3&#4\end{smallmatrix}\right)}

\newcommand{\Mn}{{\ensuremath{\operatorname{M}_n}}}

\def\emphh{\textbf}

\newcommand{\indnota}[1]{\ensuremath{#1}}

\def\sumprime{\mathop{\sum{\raise3pt\hbox{${}'$}}}}

\makeatletter
\def\revddots{\mathinner{\mkern1mu\raise\p@
\vbox{\kern7\p@\hbox{.}}\mkern2mu
\raise4\p@\hbox{.}\mkern2mu\raise7\p@\hbox{.}\mkern1mu}}
\makeatother

\providecommand{\abs}[1]{\left\vert #1 \right\vert}

\newcommand{\comment}[1]{}

\newcommand{\Sp}{\SP}

\def\twistedcomp{m}

\begin{document}

\title[Maximal Tori In Groups of Type $D_n$]{Rational Conjugacy Classes of Maximal Tori In Groups of Type $D_n$}
\author{Andrew Fiori}

\thanks{Part of this work was done while the author was a Fields postdoctoral researcher at Queen's University, further work was done at the University of Calgary with support from the Pacific Institute for Mathematical Sciences (PIMS)}

\email{andrew.fiori@ucalgary.ca}
\address{Mathematics \& Statistics
612 Campus Place N.W.
University of Calgary
2500 University Drive NW
Calgary, AB, Canada
T2N 1N4}

\begin{abstract}
We give a concrete characterization of the rational conjugacy classes of maximal tori in groups of type $D_n$, with specific emphasis on the case of number fields and p-adic fields.
This includes the forms associated to quadratic spaces, all of their inner and outer forms as well as the Spin groups, their simply connected covers. In particular, in this work we handle all (simply connected) outer forms of $D_4$.
\end{abstract}

\maketitle

\section{Introduction}

The primary goal of this work is the complete concrete classification of rational conjugacy classes of algebraic tori in groups of type $D_n$, including especially the case of $n=4$, the triality groups.

As an abstract classification in terms of Galois cohomology already exists (see \cite{WalkerThesisTori, ReederElliptic}, part of the goal of this work is to relate the concrete descriptions we shall describe with the Galois cohomology sets they describe.

Many cases of what we are looking at have already been studied, the case of pure inner forms of orthogonal groups in particular is well studied (see \cite{WalkerThesisTori, Fiori1, BayerTori, Brus_OrthTori}), and much of what we will say about this case can be deduced by combining the results of these various papers.
The simply connected forms are less well studied though some results about the spin groups can be deduced from \cite{Fiori1}.
Moreover, concrete descriptions are largely lacking for the forms which are not pure inner forms, though some results exist (see for example \cite{Brus_OrthTori, CKMFrobeniusAlgebras}).
Very little appears to be known about the triality forms of $D_4$.

Throughout this paper we shall be considering algebraic structures over a field $k$.
We shall always be assuming that $k$ does not have characteristic $2$.
Though a number of results are phrased with the implication that the field is a local or global field, such an assumption is only necessary when we make reference to an explicit classifications using cohomological invariants, where in a more general setting higher cohomological invariants may be needed.
We have tried to make it explicit when such an assumption is needed in an argument.

The sections of this paper are organized as follows:
\begin{itemize}
\item Section \ref{sec:quadratic} covers the definitions relevant to understanding the construction and classification of all groups of type $D_n$ for $n\neq 4$ (and many of them in the case $n=4$).
\item Section \ref{sec:triality} covers the definitions relevant to understanding the construction and classification of all groups of type $D_4$.
\item Section \ref{sec:tori} contains the main results of this paper and classifies the rational conjugacy classes of maximal tori in groups of type $D_n$.
\end{itemize}

Many of the results concerning $D_n$ (for $n\neq 4$) are direct generalizations from \cite{Fiori1}, which considered only a restricted class of these groups.
As many proofs proceed similarly, for the sake of brevity when possible we will use directly the results from our earlier work.

We draw the readers attention to several important expository results which may be known to experts, or otherwise appear in the literature. Specifically:
\begin{itemize}
\item Lemma \ref{lem:Covers} which gives general results about conjugacy classes of tori in covering groups.
\item Lemma \ref{lem:ToriResScalar} which gives general results about tori in restriction of scalar groups.
\item Theorem \ref{thm:lambdafortoriorth} which characterizes the rational conjugacy classes of a torus which embeds into an orthogonal group.
\item Theorems \ref{thm:invariants_upgraded} and \ref{thm:the_result_upgraded} which characterize when a torus will admit locally everywhere (but not necessarily global) embeddings into an orthogonal group.
\item Theorem \ref{thm:localglobal} which gives certain local/global criterion for when maximal tori will exist of global fields.
\end{itemize}

We also draw their attention to the major new results of this work. Specifically:
\begin{itemize}
\item Theorem \ref{thm:TORIINSPIN} which together with Definition \ref{df:psi} describes the structure of tori in Spin groups.
\item Theorem \ref{thm:CONGTORISPINCOVERS} which characterizes the conjugacy classes of tori in $\Spin$ whose images become conjugate in $\SO$.
\item Theorem \ref{thm:bigone} which characterizes the tori in simply connected groups of type $D_4$.
\item The work of Section \ref{subsec:explicitcombin} builds off this result, and allows for a more explicit understanding of some of the conditions of Theorem \ref{thm:bigone}.
\end{itemize}

\section{Quadratic Spaces and Algebras with Orthogonal Involutions}\label{sec:quadratic}

The groups we wish to consider will be associated to algebras with orthogonal involutions, as such we first introduce the relevant background.
Almost everything we are saying in this introduction comes directly, or with minor modification, from \cite{book_of_involutions}. 
We will be making frequent implicit references to results concerning the Galois cohomology of algebraic groups over number fields.
Good references for these various results include \cite{Serre_cohom}, \cite{PlatinovRapinchuk} or \cite{book_of_involutions}.

\begin{df}
By a \emphh{quadratic space} over $k$ we mean a vector space $V$ equipped with a non-degenerate symmetric bilinear pairing $B:V\times V \rightarrow k$.
The associated quadratic form is 
\[Q(x) = B(x,x). \]
\end{df}

\begin{df}
Let $A$ be a central simple algebra of degree $n$ over $k$, by an \emphh{involution} on $A$ (of the first kind) we mean a $k$-module map $\tau: A \rightarrow A$ such that:
\begin{enumerate}
\item the map $\tau$ is $k$-linear.
\item $\tau(xy) = \tau(y)\tau(x)$ for all $x,y\in A$.
\item $\tau^2(x) = x$ for all $x\in A$.
\end{enumerate}
We say that $\tau$ is \emphh{symplectic} if $\dim(A^\tau) = n(n-1)/2$ and \emphh{orthogonal} if $\dim(A^\tau)=n(n+1)/2$.
By \cite[Prop 2.6]{book_of_involutions} these are the only two options.

Let $E$ be an \'etale algebra of dimension $2n$ over $k$, by an \emphh{involution} on $E$ we mean an automorphism $\sigma$ of $E$ of order $2$ such that:
$\dim(E^\sigma) = n$.
\end{df}

\begin{thm}
Let $A$ be a central simple algebra of degree $n$ over $k$ then $A$ has an involution over $k$ if and only if $A$ is isomorphic to its opposite algebra, in particular, if and only if $A$ is a matrix algebra or a matrix algebra over a quaternion algebra.
\end{thm}
\cite[Thm. 3.1]{book_of_involutions}.
\medskip

\begin{ex}~
\begin{itemize}
\item The transpose involution on $M_2(k)$ is an orthogonal involution.
\item The standard involution $x\mapsto \overline{x}$ on a quaternion algebra $A$ is a symplectic involution.
\item The involution $\smtx abcd \mapsto \smtx d{-b}{-c}a$ is a symplectic involution of $M_2(k)$. This is the standard involution on a $M_2(k)$ viewed as a quaternion algebra.
\item Let $(V,B)$ be a quadratic space, then the adjoint involution $\Ad_B$ on $\End_V$ is an orthogonal involution.
\end{itemize}
\end{ex}

\begin{thm}
We have the following characterization of the simple algebras with orthogonal and symplectic involutions over a field $k$.
\begin{itemize}
\item Let $M$ be the matrix algebra $M_n(k)$ and let $t$ denote the transpose involution.
\begin{itemize}
\item
         Let $\gamma\in M$ be such that $\gamma^t = -\gamma$ then the involution $\tau(x) = \gamma (x^t)\gamma^{-1}$ is symplectic.
         Moreover, all symplectic involutions on $M$ arise this way.
\item
        Let $\gamma\in M$ be such that $\gamma^t = \gamma$ then the involution $\tau(x) = \gamma (x^t)\gamma^{-1}$ is orthogonal.
        Moreover, all orthogonal involutions on $M$ arise this way.
\end{itemize}
\item Let $A$ be a quaternion algebra with standard involution $\overline{\cdot}$, and let $M$ be the algebra $M_n(A)$ with involution $\tau(x)=\overline{x}^t$.
\begin{itemize}
\item
         Let $\gamma\in M$ be such that $\tau(\gamma) = -\gamma$ then the involution $\tau_\gamma(x) = \gamma \tau(x)\gamma^{-1}$ is orthogonal.
         Moreover, all orthogonal involutions on $M$ arise this way.
\item
         Let $\gamma\in M$ be such that $\tau(\gamma) = \gamma$ then the involution $\tau_\gamma(x) = \gamma \tau(x)\gamma^{-1}$ is symplectic.
         Moreover, all symplectic involutions on $M$ arise this way.
\end{itemize}
\item Let $A$ be a central simple algebra over $k$ algebra with two symplectic involutions $\tau_1$ and $\tau_2$.
         Then $(A,\tau_1)$ and $(A,\tau_2)$ are isomorphic as algebras with involutions.
\item  Let $A$ be a central simple algebra over $k$ algebra with two orthogonal involutions $\tau_1$ and $\tau_2$.
         Then $(A,\tau_1)$ and $(A,\tau_2)$ are isomorphic as algebras if and only if there exists $\gamma\in A$ such that:
          $\tau_2( \gamma x \gamma^{-1}) = \gamma \tau_1(x) \gamma^{-1}$
\end{itemize}
\end{thm}
\cite[Prop 2.20, 2.22]{book_of_involutions}.

\begin{nota}
Given an involution $\tau$ on a central simple algebra $A$ and an element $g\in A^\times$ such that $\tau(g) = \pm g$ we shall denote by:
\[ \tau_g = \Int_g \circ \tau \]
the involution taking $x$ to $g \tau(x) g^{-1}$.

We note that $\tau_g$ is not necessarily isomorphic to $\tau$ as isomorphisms have the form:
\[ \Int_g \circ \tau \circ \Int_{g^{-1}} = \tau_{g\tau(g)}. \]

We remark further that $\tau$ and $\tau_g$ have the same type (symplectic or orthogonal) if $\tau(g) = g$, otherwise they have different types. Moreover, by the above theorem all involutions arise this way.
\end{nota}

\begin{cons}\label{cons:orthogroup}
Let $(A,\tau)$ be a central simple algebra of degree $n$ with orthogonal involution $\tau$.
We define the associated orthogonal group to the group scheme $\Orth_{A,\tau}$ whose functor of points is:
\[ \Orth_{A,\tau}(R) = \{ g \in (A\otimes_k R)^\times \mid \tau(g)g = 1 \}. \]
Further, we define the special orthogonal group $\SO_{A,\tau}$ to be:
\[ \SO_{A,\tau}(R) = \{ g \in (A\otimes_k R)^\times \mid \tau(g)g = 1 \text{ and } N_{A/k}(g) = 1 \} \]
where $N_{A/k}$ denotes the reduced norm.

If $n$ is odd this group is of type $B_{(n-1)/2}$, if $n$ is even it is of type $D_{n/2}$.

For the remainder of this paper we shall be largely focused on the case where $n$ is even.
\end{cons}

\begin{rmk}
Given a quadratic space $(V,B)$ then taking $(A,\tau)$ to be $(\End_V,\Ad_B)$ we obtain the usual orthogonal and special orthogonal groups associated to $(V,B)$.
\end{rmk}

\subsection{Invariants of Algebras with Orthogonal Involutions}

We now discuss some of the invariants of algebras with orthogonal involutions. The invariants we shall discuss are sufficient to classify the groups under consideration for both local and global fields. In other settings other invariants may be needed for this purpose. Aside from being relevant in classifying the groups, the invariants we shall focus on will be directly connected to the classification of tori in these groups.  

The invariants we will discuss are the discriminant, the Clifford invariant and the index.

\subsubsection{Discriminants}

\begin{thm}\label{thm:discriminant}
Let $(A,\tau)$ be a central simple algebra of degree $2n$ over $k$ with orthogonal involution $\tau$.
Let $x\in A^\times$ be such that $\tau(x)=-x$, then $N_{A/k}(x)$ is independent of the choice of $x$ when viewed as an element of $k^\times/(k^\times)^2$.

We shall call $N_{A/k}(x)$ the \emphh{discriminant} of $\tau$, and note that the discriminant of the adjoint involution is the discriminant of the associated quadratic form. The discriminant will be denoted $D(\tau)$.
\end{thm}
\cite[Prop 7.1]{book_of_involutions}.

\subsubsection{Clifford Algebras}

For an explicit construction of the Clifford algebra we refer the reader to \cite[Ch. 9]{book_of_involutions} we include here the details we shall need.
\begin{cons}\label{cons:clifford}
Let $(A,\tau)$ be a central simple algebra of degree $2n$ over $k$ with orthogonal involution.
Let $T(A)$ denote the tensor algebra of $A$ (viewed as a $k$-module). There is an ideal $J(A,\tau)$ for $T(A)$ such that the Clifford algebra is given by:
\[ \Clif^+_{A,\tau} = \frac{T(A)}{J(A,\tau)}.\]
\end{cons}

The important features of this construction are summarized in the following theorem and remark. 
These results follow from \cite[Thm 9.12]{book_of_involutions}.
\begin{thm}\label{thm:structcliff}
Let $(A,\tau)$ be a central simple algebra of degree $2n$ over $k$ with orthogonal involution.
\begin{itemize}
\item The Clifford algebra of $(A,\tau)$ is a central simple algebra over $F= k[X]/(X^2-(-1)^n D(\tau))$. Recall $D(\tau)$ is the discriminant of $(A,\tau)$.
         When $n$ is even it is in the $2$-torsion of the Brauer group, when $n$ is odd it is in the $4$-torsion.
\item The involution $\tau$ on $A$ induces an ``involution'' $\overline{\tau}$ on $\Clif^+_{A,\tau}$ given by:
\[ \overline{\tau}(x_1\otimes\cdots\otimes x_r) = \tau(x_r)\otimes\cdots\otimes\tau(x_1). \]
\begin{itemize}
\item $\overline{\tau}$ restricts to a non-trivial involution of the center of  $\Clif^+_{A,\tau}$ if $n$ is odd.
\item $\overline{\tau}$ is orthogonal over $F$ if $n\cong 0 \pmod{4}$ and is symplectic if $n\cong 2 \pmod{4}$.
\end{itemize}
\end{itemize}
\end{thm}

\begin{rmk}\label{rmk:CliffordStructure}
Over a local field, the Clifford algebra has the following structure:
\begin{itemize}
\item If $n$ is even, $A$ is a matrix algebra, and the discriminant is non-trivial, then class of $[\Clif^+_{A,\tau}]$ is trivial.
\item If $n$ is even, $A$ is not a matrix algebra, and the discriminant is non-trivial then $[\Clif^+_{A,\tau}]$ is non-trivial.
\item If $n$ is even, $A$ is a matrix algebra, and the discriminant is trivial, then class of $[\Clif^+_{A,\tau}]$ is a direct sum of two isomorphic algebras, their triviality depends on the choice of $\tau$.
\item If $n$ is even, $A$ is not a matrix algebra, and the discriminant is trivial then $[\Clif^+_{A,\tau}]$ is a direct sum of two non-isomorphic algebras in the $2$-torsion of the Brauer group, the isomorphism over the center depends on the choice of $\tau$, note that the isomorphism class over $k$ does not.

\item If $n$ is odd, $A$ is a matrix algebra, and the discriminant is non-trivial, then class of $[\Clif^+_{A,\tau}]$ is trivial.
\item If $n$ is odd, $A$ is not a matrix algebra, and the discriminant is non-trivial then $[\Clif^+_{A,\tau}]$ is trivial.
\item If $n$ is odd, $A$ is a matrix algebra, and the discriminant is trivial, then class of $[\Clif^+_{A,\tau}]$ is a direct sum of two isomorphic algebras in the $2$-torsion of the Brauer group, their triviality depends on the choice of $\tau$.
\item If $n$ is odd, $A$ is not a matrix algebra, and the discriminant is trivial then $[\Clif^+_{A,\tau}]$ is a direct sum of two non-isomorphic algebras which are $4$-torsion, the isomorphism over the center depends on the choice of $\tau$, note that the isomorphism class over $k$ does not.
\end{itemize}
The structure over a global field can then be deduced from this, as the global structure is determined by the local structure.
\end{rmk}

In order to avoid certain auxilliary constructions usually necessary to define the spin group we shall use several results from \cite[Sec. 13.A]{book_of_involutions}.
\begin{prop}\label{prop:AtoCliff}
There is a natural map $C: \Aut(A,\tau) \rightarrow \Aut(\Clif^+_{A,\tau})$ induced from the natural action of $ \Aut(A,\tau)$ on $T(\underline{A})$.
By an abuse of notation we shall also denote the map:
 \[C\circ \Inn : \Orth_{A,\tau} \rightarrow \Aut(\Clif^+_{A,\tau}),\]
 where $\Inn$ denotes the inner automorphism, by $C$.
\end{prop}

It follows from \cite[Prop 13.5]{book_of_involutions} that we may make the following (non-standard) definition of the Spin group.
\begin{cons}\label{cons:spinAtau}
We define:
\[  \Spin_{A,\tau} = (\SO_{\Clif^+_{A,\tau},\tau} \times_{\Aut(\Clif^+_{A,\tau})} \Aut(A,\tau))^0. \]
\end{cons}

\begin{rmk}
Note that in the above $\SO_{\Clif^+_{A,\tau},\tau}$ might be more correctly expressed as one of $\SU_{\Clif^+_{A,\tau},\tau}$ or $\Sp_{\Clif^+_{A,\tau},\tau}$ depending on the degree of $A$. 
The definition of $\SO_{A,\tau}$, $\SU_{A,\tau}$ and $\Sp_{A,\tau}$ are essentially indistinguishable except for the conditions on the type of involution $\tau$. Theorem \ref{thm:structcliff} describes the various possibilities of $\tau$.
\end{rmk}

As usual there exists a map from the spin group to the special orthogonal group.
\begin{prop}\label{prop:spinAtau}
The natural map $ \Spin_{A,\tau}  \rightarrow  \Aut(A,\tau)$ factors through a map:
\[ \Spin_{A,\tau} \overset{\chi}\rightarrow  \SO_{A,\tau} \rightarrow \Aut(A,\tau). \]
Moreover, the map $C \circ \chi$ agrees with the usual map from $\Clif^+_{A,\tau}$ to $\Aut(\Clif^+_{A,\tau})$.
\end{prop}
See \cite[Sec. 13.A]{book_of_involutions}.

\subsubsection{Index}

In order to obtain a complete set of invariants for real fields and consequently global fields we shall need one further invariant.

\begin{df}\label{df:index}
Let $(A,\tau)$ be a central simple algebra of degree $n$ over $k$ with orthogonal involution $\tau$.
A right ideal $I$ of $A$ is said to be isotropic if $x\tau(y) = 0$ for all $x,y\in I$.

The index of $A$ is the set:
\[ {\rm ind}(A,\tau) = \{ \dim(I) \mid I \text{ isotropic} \}. \]

Note that if $A$ is a matrix algebra over a division algebra $D$ of degree $m$ then:
 \[ {\rm ind}(A,\tau) = \{ 0, m, \ldots, \ell m \} \]
 for some integer $\ell$ (see \cite[6.3]{book_of_involutions}). We shall also refer to $\ell$ as the index.

Let $\nu$ denote a place of $k$, we shall denote by $\ell_\nu$ the index of $(A,\tau)$ at $k_\nu$.
\end{df}

\subsection{Cohomological Interpretation and Classification over Local and Global Fields}\label{subsec:cohominterp}

The following theorem essentially classifies groups of type $D_n$ for $n\neq 4$ over local and global fields.

\begin{thm}
Let $k$ be a local or global field.
Let $\SO_{2n}$ denote the standard form of an orthogonal group over $k$ for a quadratic space of dimension $2n$.
\begin{itemize}
\item Forms of $\SO_{2n}$ are all of the form given by Construction \ref{cons:orthogroup}.
\item If $k$ is a global or local field, then the following are a complete set of invariants:
\begin{itemize}
\item The discriminant $\delta$ (Theorem \ref{thm:discriminant}) as an element of $k^\times/(k^\times)^2$.
\item The even Clifford invariant (Construction \ref{cons:clifford}) as an element of the Brauer group of the algebra $k[X]/(X^2-(-1)^n\delta)$.
\item The indices $\ell_\nu$ (Definition \ref{df:index}) at the real places $\nu$ of $k$.
\end{itemize}
\end{itemize}
Moreover, for $n\neq 4$, forms of $\Spin_{2n}$ and $\PSO_{2n}$ are classified by the same data.
\end{thm}
\cite[Thm. 26.15 and Sec. 31]{book_of_involutions}

\begin{rmk}
The case of $\Spin_8$ and $\PSO_8$ will be discussed in the coming sections.

The classification above is essentially equivalent to the classification of central simple algebras with involutions, we note that it is not equivalent to the classification of quadratic forms, in particular the Hasse invariant is not always determined by the even Clifford invariant.
\end{rmk}

We briefly sketch the Galois cohomological interpretation of these results. For a more detailed discussion see \cite{book_of_involutions}.

Let $(A,\tau)$ be a central simple algebra of degree $2n$ over $k$ with orthogonal involution $\tau$.
There are exact sequences:
\[ 1 \rightarrow \SO_{A,\tau} \rightarrow \Orth_{A,\tau} \rightarrow \{\pm1\} \rightarrow 1 \]
and
\[ 1 \rightarrow \Aut_{A,\tau}^0 \rightarrow \Aut_{A,\tau} \rightarrow \{ \pm 1\} \rightarrow 1 \]
and
\[ 1 \rightarrow Z(\Spin_{A,\tau}) \rightarrow \Spin_{A,\tau} \rightarrow \Aut_{A,\tau}^0 \rightarrow 1 \]
and
\[ 1 \rightarrow \{\pm 1\} \rightarrow \Spin_{A,\tau} \rightarrow \SO_{A,\tau}\rightarrow 1. \]
These lead to exact sequences in Galois cohomology.

We have the following:
\begin{itemize}
\item  The set $H^1(\Gal(\overline{k}/k),\Aut_{A,\tau})$ gives a classification of either pairs $(A,\tau)$ of central simple algebras of degree $2n$ with an orthogonal involution or of isomorphism classes of orthogonal groups being the automorphism group of both.
\item When $A$ is the matrix algebra over a division algebra $D$ then
 $\Orth_{A,\tau}$ is the automorphism group of a quadratic space over $D$.
         Moreover,
  \[ H^1(\Gal(\overline{k}/k),\Orth_{A,\tau})\] 
classifies these quadratic spaces.
         That is, it is classifying elements $x$ of $A$ such that $\tau(x) = x$, up to equivalence $x\sim \tau(y)x y^{-1}$ for $y\in A^\times$. The map
\[ H^1(\Gal(\overline{k}/k),\Orth_{A,\tau}) \rightarrow  H^1(\Gal(\overline{k}/k), \{\pm1\}) = k^\times/(k^\times)^2\]
 gives the discriminant of the form, that is the discriminant of the symmetric element defining it, hence
          \[ \im\left(H^1(\Gal(\overline{k}/k),\SO_{A,\tau}) \rightarrow H^1(\Gal(\overline{k}/k),\Orth_{A,\tau})\right)\]
classifies forms with the same discriminant as $\tau$.

When  $A$ is a central simple algebra we may explicitly describe $H^1(\Gal(\overline{k}/k),\SO_{A,\tau})$ as a torsor via the bijection:
\[ H^1(\Gal(\overline{k}/k),\SO_{A,\tau}) \simeq \{ (s,z) \in A^\times\times k^\times \;\mid\; \tau(s)=s\text{ and } N_{A/k}(s) = z^2 \}/\sim \]
where the equivalence relation on the right is given by $(s',z') \sim (s,z)$ if there exists $a\in A^\times$ with $s'= as\tau(a)$ and $z' = N_{A/k}(a)z$ (see \cite[Sec 29.D eqn 29.27]{book_of_involutions}.

The map: 
\[ H^1(\Gal(\overline{k}/k),\SO_{A,\tau}) \rightarrow H^1(\Gal(\overline{k}/k),\Aut_{A,\tau}) \]
associates to $(s,z)$ the algebra with involution $(A,\tau_s)$ whereas the map
\[ H^1(\Gal(\overline{k}/k),\SO_{A,\tau}) \rightarrow H^1(\Gal(\overline{k}/k),\Orth_{A,\tau}) \]
associates to $(s,z)$ simply the element $s$.
Moreover, the group $H^0(\Gal(\overline{k}/k),\mu_2)$, which is the kernel of the above map, acts on pairs $(s,z)$ by $z\mapsto -z$. Finally, we remark that if $A$ is a matrix algebra, these elements will be equivalent.

\item As above we find that:
\[H^1(\Gal(\overline{k}/k), \Aut_{A,\tau}) \rightarrow H^1(\Gal(\overline{k}/k),\{ \pm 1\})\] allows us to associate the discriminant of the involution, and, as above       
           \[ \im\left(H^1(\Gal(\overline{k}/k),\Aut_{A,\tau}^0) \rightarrow H^1(\Gal(\overline{k}/k),\Aut_{A,\tau}) \right)\] 
classifies pairs $(A',\tau')$ where $\tau'$ has the same discriminant as $\tau$. 
Likewise, as above, the additional information captured by an element of $H^1(\Gal(\overline{k}/k),\Aut_{A,\tau}^0) $ is again described by considering pairs $(\tau',z)$ where $z\in k^\times$ is an element whose square is the norm of the discriminant of $\tau'$ with the same equivalence as before.

\item  If $n$ is even then the center of the spin group is:
\[  Z(\Spin_{A,\tau}) \simeq \Res_{F/k}(\mu_{2,F}) \] where $F$ denotes the center of the even Clifford algebra and thus we obtain:
          \[ H^1(\Gal(\overline{k}/k),\Spin_{A,\tau})) \rightarrow H^1(\Gal(\overline{k}/k), \Aut_{A,\tau}^0) \rightarrow H^2(\Gal(\overline{k}/k),\Res_{F/k}(\mu_{2,F})). \]
          The set 
\[ H^2(\Gal(\overline{k}/k),\Res_{F/k}(\mu_{2,F})) \] gives the $2$-torsion of the Brauer group of $F$, and gives the Clifford invariant.
          Over a local or global field the set \[ H^1(\Gal(\overline{k}/k),\Spin_{A,\tau})) \] is supported at the real places of $k$, as the group $\Spin_{A,\tau}$ is simply connected.
          One may check that modulo the action of the image of \[ H^1(\Gal(\overline{k}/k),\Res_{F/k}(\mu_{2,F})) \simeq F^\times/ (F^\times)^2 \] a class here determines the index.
We remark that the action of $F^\times/ (F^\times)^2$ on pairs $(s,z)$ is by rescaling $(s,z) \mapsto (fs,f^{n/2}s)$.

\item If $n$ is odd then the center of the spin group is:
\[ Z(\Spin_{A,\tau}) \simeq \mu_{4,F}^1 =  \Res_{F/k}(\{ g \in \mu_{4,F} \mid N_{F/k}(g) = 1 \})\] where $F$ denotes the center of the even Clifford algebra and thus we obtain:
\[ H^1(\Gal(\overline{k}/k),\Spin_{A,\tau})) \rightarrow H^1(\Gal(\overline{k}/k), \Aut_{A,\tau}^0) \rightarrow H^2(\Gal(\overline{k}/k),\Res_{F/k}(\mu_{2,F})). \]
        The set \[ H^2(\Gal(\overline{k}/k),\Res_{F/k}\mu_{4,F}^1))\] gives the kernel of the corestriction map from $F$ to $k$ of the $4$-torsion of the Brauer group of $F$, and again, gives the Clifford invariant.
        Note that the kernel of the corestriction map is trivial for local fields, thus the kernel is supported on places where $F$ splits. 

          For local and global fields the set 
\[ H^1(\Gal(\overline{k}/k),\Spin_{A,\tau}) \] 
is supported at the real places of $k$, as the group $\Spin_{A,\tau}$ is simply connected.
  
        One may check that modulo the action of the image of
                \[ H^1(\Gal(\overline{k}/k),\Res_{F/k}(\mu_{2,F})) \simeq F^\times/ (F^\times)^2\]
          a class here determines the index. Once again, the action of the $ F^\times/ (F^\times)^2$ on quadratic forms can be interpretted as rescalling them. 
Note that this action controls the discriminant of the form, which agrees with the observation that the center lands outside the connected component of the identity.

\item In a manner similar to the above, we may obtain cohomological invariants of quadratic forms, if $(A,\tau)$ is associated to the endomorphism algebra of a quadratic space over a division ring, the sequence:
\[ 1 \rightarrow \{\pm 1\} \rightarrow \Spin_{A,\tau} \rightarrow \SO_{A,\tau}\rightarrow 1 \]
associates to a quadratic form an element of the Brauer group of $k$, this sequence gives the Witt invariant of the quadratic form and a signature at real places.

We remark that here we are obtaining an element of the Brauer group of $k$ rather than $F$. This agrees with the fact that a quadratic form is a more specific piece of information than its associated involution. This difference can also be seen in a comparison also of Theorems \ref{thm:brus_res} and \ref{thm:invariants_upgraded} where we see that passing from the one invariant to the other is done by restriction, which may have a kernel, and is not typically surjective.

\end{itemize}

\begin{rmk}
Finally, we should remark that in performing the analysis, the order of twisting that we consider is:
\begin{enumerate}
\item Pick a discriminant.
\item Pick a Clifford invariant.
\item Pick an index.
\end{enumerate}
That is to say, we may always pick an arbitrary discriminant, this choice then determines in which Brauer group the Clifford invariant lives.
We may then pick the Clifford invariant, which will determine the algebra $A$ (and strongly restrict the choice of $\tau$) noting that the choice of $\tau$ was specifically restricted by the choice of discriminant.
The choice of Clifford invariant then limits the choices for which indexes can be chosen.
\end{rmk}

\section{Triality and Groups of type $D_4$}\label{sec:triality}

In order to extend the above classification to better handle the case of $D_4$ groups, we shall need to introduce the structures for which these groups will be the automorphism groups, this requires that we discuss triality.
In this section we will sequentially consider several types of algebras and their automorphism groups.
Each of the structures of consideration will allow us to define a group of type $D_4$ as either its automorphism group or the simply connected component of its automorphism group.
In this way we shall be obtaining different classes of pure inner (or in some cases outer) forms of groups of type $D_4$.

The classes we obtain and the source of the forms are summarized below:
\begin{enumerate}
\item Symmetric Composition algebras are associated to classes from $H^1(\Gal(\overline{k}/k),\Aut(S) )$, noting that the type of $\Aut(S)$ is not unique.
\item Twisted Compositions  associated to classes from $H^1(\Gal(\overline{k}/k),S_3 \ltimes \Spin)$.
\item Trilitarian algebras associated to classes from $H^1(\Gal(\overline{k}/k),S_3 \ltimes \PSO)$.
\end{enumerate}
The last of which allows us to obtain all the groups of type $D_4$. The reason for introducing the other classes is to allow us to construct examples of the later ones, but also because the structure of tori in the earlier cases is simpler and it is desirable to be able to study these special cases more concretely.

Most of the content of this section can be found in \cite{book_of_involutions}, a more thorough treatment can be found there.

\subsection{Symmetric Compositions}

\begin{df}
A (regular) \emphh{composition algebra} $C = (C,\cdot,N)$ over $k$ is a $k$-algebra (not necessarily associative, commutative, or with identity) with a regular quadratic form $N : C \rightarrow k$ satisfying 
\[ N(xy) = N(x)N(y). \]
Denote the associated bilinear form $B(x,y) = N(x+y) - N(x) - N(y)$.

We call a composition algebra $(C,\cdot,N)$ \emphh{symmetric}, if $B(x\cdot y, z) = B(x, y\cdot z)$.
\end{df}

\begin{prop}
A composition algebra is symmetric if and only if $x\cdot (y\cdot x) = N(x) y = (x\cdot y)\cdot x$.
\end{prop}
\cite[34.1]{book_of_involutions}.

\begin{ex}\label{ex:parahurwitz}
Let $(B,\sigma)$ be an octonion algebra and set $N(x) = x\sigma(x)$, define an algebra $ (B,\ast,N)$ by:
\[ x \ast y = \sigma(x) \cdot \sigma(y). \]
The algebra $(B,\ast,N)$ is a symmetric composition algebra.
\end{ex}

\begin{rmk}
Not all symmetric composition algebras are constructed as above.
There are a small number of families of constructions, see \cite[34.7]{book_of_involutions} for a classification result.
We are primarily interested in the case of dimension $8$.

Note that as with the example above, most symmetric composition algebras are not unital.
\end{rmk}

\begin{prop}\label{prop:map_of_cliff}
Let $(S,\cdot,N)$ be a symmetric composition algebra of dimension $8$ over $k$.

For all $\lambda \in k$ there exists a map $S \rightarrow \End_k(S\oplus S)$ given by:
\[ x \mapsto \mtx{0}{\lambda\ell_x}{r_x}{0} \]
which induces isomorphisms 
\[ \alpha_S : \Clif_{\lambda N} \rightarrow (\End_k(S\oplus S),\sigma_{N\oplus N}) \qquad\text{and}\qquad
\alpha_S : \Clif^+_{\lambda N} \rightarrow (\End_k(S),\sigma_{N}) \oplus (\End_k(S),\sigma_{N}). \]
 Where here $\Clif_{\lambda N}$ denotes the Clifford algebra associated to quadratic space with quadratic form $\lambda N$.
\end{prop}
\cite[35.1]{book_of_involutions}.

\begin{prop}
Let $(C,\cdot,N)$ be a symmetric composition algebra of dimension $8$ over $k$.
Let $g$ be an isometry of $(C,N)$ then there exists isometries $g^+$ and $g^-$ of $N$ such that:
\[ g(x\cdot y) = g^+(y)\cdot g^-(x), \qquad  g^+(x\cdot y) = g(y)\cdot g^-(x)  \quad\text{and}\quad  g^-(x\cdot y) = g^+(y)\cdot g(x)  \]
moreover any one of these identities implies the others.
\end{prop}
\cite[35.4]{book_of_involutions}.

\begin{df}
Let $(C,\cdot,N)$ be a regular symmetric composition algebra of dimension $8$ over $k$.

We define the associated spin group to be the group scheme $G$ whose points over $R$ are:
\[ \Spin_{(C,\cdot,N)}(R) = \{ (g,g^+,g^-) \in \Orth_N^3 \mid g(x\cdot y) = g^+(y)\cdot g^-(x) \} \]
\end{df}

\begin{thm}
There is an isomorphism:
\[ \Spin_{(C,\cdot,N)}\simeq  \Spin_{(C,N)}. \]
In particular $\Spin_{(C,\cdot,N)}$ is a group of type $D_4$, and moreover the above descriptions realizes $S_3$ automophism group.
\end{thm}
\cite[35.C]{book_of_involutions}.

\subsection{Twisted Compositions}

Much of the material of this section comes from \cite[Ch. 36]{book_of_involutions}. A more thorough treatment of it can be found there.

\begin{df}
Let $L$ be a cubic \'etale algebra over $k$.
Denote the norm map of $L$ by $N_{L/k}$, the trace map $T_{L/k}$.
Let $(V,Q)$ be a regular quadratic space over $L$.
Let $\twistedcomp: V\rightarrow V$ be a quadratic map.
We call a quadruple $(L,V,Q,\twistedcomp)$ a \emphh{twisted composition} if it satisfies the following conditions:
\begin{enumerate}
\item $\ell\twistedcomp(\ell v) = N_{L/k}(\ell)\twistedcomp(v)$
\item $Q(v)Q(\twistedcomp(v)) = N_{L/k}(Q(v))$.
\end{enumerate}
\end{df}

\begin{ex}\label{ex:split_twisted}
Let $C=(C,\ast,N)$ be a symmetric composition algebra over $k$.Consider:
\[  L=k\times k \times k,\qquad V=C\otimes_k L \simeq C\times C \times C, \]
 and $Q$ be the natural extension of $N$ to $V$.
Define $\twistedcomp :  V  \rightarrow V$ by 
\[ \twistedcomp(x,y,z) = (v_1\ast v_2, v_2\ast v_0,v_0\ast v_1). \]
Then $(L,V, Q,\twistedcomp)$ is a twisted composition. \cite[36.2]{book_of_involutions}
\end{ex}

\begin{prop}
If $(L,V,Q,\twistedcomp)$ is a twisted composition algebra, then for all $\lambda\in L^\times$ so is 
\[ \left(L,V,\lambda^{-1}N_{L/k}(\lambda)Q,\lambda\twistedcomp\right). \]
\end{prop}
\cite[36.1]{book_of_involutions}

\begin{prop}\label{prop:split_twisted}
Let $(C,\cdot,N)$ be a symmetric composition algebra and associate $(L,V,Q,\twistedcomp)$ as above.
The for all $\lambda \in L^\times$ we have that $(L,V,\lambda^{-1}N_{L/k}(\lambda)Q, \lambda\twistedcomp )$ is a twisted composition algebra.
Moreover, all twisted composition algebras over $L=k\times k \times k$ arise this way.
\end{prop}
\cite[36.3]{book_of_involutions}

What the above proposition says is that all twisted compositions over the split algebra $L$ arise from rescalings of symmetric compositions.
\begin{ex}\label{ex:cyclic_comp}
Let $L$ be a cubic \'etale algebra with $\rho$ an element of order $3$ of $\Aut(L/k)$.
Let $C=(C,\ast,N)$ be a symmetric composition algebra.
Let $\twistedcomp : C \otimes_k L \rightarrow C \otimes_k L$ be the map $x\otimes\ell \mapsto (x\otimes\rho(\ell)) \ast (x\otimes\rho^2(\ell)$.
Then $(L,C,N,\twistedcomp)$ is a twisted composition. \cite[36.11]{book_of_involutions}
\end{ex}

\begin{ex}\label{ex:twisted_hurwitz}
Let $(B,\sigma)$ be an octonion algebra and $C=(C,\ast,N)$ be the associated symmetric composition algebra over $k$ as in Example \ref{ex:parahurwitz}.
Fix generators $\rho$ of $\Aut(L\otimes_k\Delta_L / \Delta_L)$ and $\iota$ of $\Aut(\Delta_L/k)$ and
let $\tilde{\twistedcomp}$ be the twisted composition over $(C \otimes_k L \otimes_k \Delta_L)$ as in Example \ref{ex:cyclic_comp}.
Let $\tilde{\iota} : C \otimes_k L \otimes_k \Delta_L \rightarrow C \otimes_k L \otimes_k \Delta_L$ be the map $\sigma\otimes 1 \otimes \iota$ and set
 \[ V = \{ x\in C \otimes_k L \otimes_k \Delta_L \mid \tilde{\iota}(x) = x \}. \]
Then $\tilde{\twistedcomp}$ restricts to a map $\twistedcomp: V\rightarrow V$ and $(L,V,N,\twistedcomp)$ is a twisted composition over $L$. \cite[36.C]{book_of_involutions}.
\end{ex}

\begin{prop}\label{prop:TwistedNotField}
Let $L/k$ be a cubic \'etale algebra which is not a field, then every twisted composition algebra over $L$ is of the form:
\[ (L,V,\lambda^{-1}N_{L/k}(\lambda)N, \lambda\twistedcomp ) \]
where $(L,V,N,\twistedcomp)$ is as in Example \ref{ex:twisted_hurwitz} and $\lambda \in L^\times$.
\end{prop}
\cite[36.29]{book_of_involutions}

\begin{rmk}
When $L$ is a field, there typically exist twisted composition algebras over $L$ not arising from the above construction.
\end{rmk}

\begin{thm}
Let $(L,V,Q,\twistedcomp)$ be the twisted composition of dimension $8$ associated to a symmetric composition algebra $(C,\cdot,N)$ then
 \[ \Aut(L,V,Q,\twistedcomp) \simeq S_3 \ltimes \Spin_{(C,\cdot,N)}, \]
the action of $S_3$ is via the outer automorphisms as above.

Moreover, twisted compositions are classified by $H^1(\Gal(\overline{k}/k),\Aut(L,V,Q,\twistedcomp))$.
\end{thm}
\cite[36.5 and 36.7]{book_of_involutions}

\begin{rmk}\label{rmk:simpletwisted}
If $L/k$ is not a field, then $L = k \oplus F$, and $V$ has a correspond decomposition $V = V_k \oplus V_F$.
The spin group $\Spin_{(L,V,Q,\twistedcomp)} = \Aut(L,V,Q,\twistedcomp)^0$ is isomorpic to the spin group of $V_k$ and $V_F$ is simply the spin representation of this group.
\end{rmk}

\begin{cor}\label{cor:trivinvars}
Let $(L,V,Q,\twistedcomp)$ be the twisted composition of dimension $8$, then as a quadratic space over $L$ the form $Q$ has trivial Clifford invariant and the discriminant of $Q$ is the discriminant of $L$.
\end{cor}
\begin{proof}
By Remark \ref{rmk:simpletwisted} the result is immediate if $L$ is not a field as in this case the discriminant of $L$ is that of $F$ which is also the discriminant of the center of the Clifford algebra.
Moreover, as the cohomological interpretation of the Clifford invariant is invariant under twisting by $\Spin$ they all share the same Clifford invariant, which is trivial by example.

Suppose now $L$ is a field, as $L$ is cubic, base change to $L$ is injective on discriminants and Clifford invariants.
After base change to $L$ we have that
\[ L \otimes_k L \simeq L \oplus (L\otimes \Delta_L) \] and thus
\[  (L,V,Q,\twistedcomp)\otimes_k L \simeq (L \oplus (L\otimes \Delta_L), V \oplus V_{(L\otimes_k \Delta_L)}, Q \oplus Q_{(L\otimes_k \Delta_L)} ,\twistedcomp) \] so that we are back in the case of Remark \ref{rmk:simpletwisted} (replacing $k$ by the field $L$). We note the important fact that the relevant quadratic space for the field factor $L$ of $L\otimes_k L$ is simply $V$ and thus we conclude as above that the discriminant is $\delta_L$ and the Clifford invariant is trivial.
\end{proof}

\begin{rmk}
It is a consequence of the above that if $L$ is not a field we still have:
 \[ \Spin_{(L,V,Q,\twistedcomp)} \otimes_k L \simeq \Spin_{V,Q}. \]
\end{rmk}

\begin{prop}
Let $(L,V,Q,\twistedcomp)$ be a twisted composition of dimension $8$.
The center $Z$ of $\Aut(L,V,Q,\twistedcomp)$ arises from the exact sequence:
\[ 1 \rightarrow Z \rightarrow \Res_{L/k}(\mu_{2,L}) \overset{N_{L/k}}\rightarrow \mu_{2,k} \rightarrow 1. \]
\end{prop}
See proof of \cite[Lem 44.14]{book_of_involutions}.

\begin{prop}\label{prop:CENTERONTWISTED}
Consider the the exact sequence:
\[ H^1(\Gal(\overline{k}/k),Z) \rightarrow H^1(\Gal(\overline{k}/k),\Aut(L,V,Q,\twistedcomp))) \rightarrow H^1(\Gal(\overline{k}/k),\Aut(L,V,Q,\twistedcomp)^{\rm adj}). \]
The group 
\[ H^1(\Gal(\overline{k}/k),Z) \simeq L^\times/k^\times(L^\times)^2\]
 acts by
 \[ (L,V,Q,\twistedcomp ) \rightarrow (L,V,\lambda^{-1}N_{L/k}(\lambda)Q, \lambda\twistedcomp ). \]
Moreover, $(L,V,Q,\twistedcomp ) \simeq(L,V,\lambda^{-1}N_{L/k}(\lambda)Q, \lambda\twistedcomp )$ if and only if $Q\simeq \lambda^{-1}N_{L/k}(\lambda)Q$.
\end{prop}
\cite[36.9 and 36.10]{book_of_involutions}.

\begin{ex}\label{ex:twistedsig}
As there are only two cubic \'etale algebras over $\bR$ neither of which is a field we can easily use Proposition \ref{prop:TwistedNotField} to describe all of the twisted compositions over $\bR$.
\begin{itemize}
\item When $L \simeq \bR \times \bR \times \bR$ we have from Proposition \ref{prop:TwistedNotField} that every twisted composition over $L$ is of the form:
\[ \bO \times \bO \times \bO \]
where $\bO$ is either the split or non-split octonions.
The quadratic form being either the standard form for the octonions, or
in the case of the non-split octonions we may replace any $2$ of these by their negatives.
\item When $L \simeq \bR \times \bC$ we have from Proposition \ref{prop:TwistedNotField} that every twisted composition over $L$ is of the form:
\[ V \times (V\otimes \bC) \]
The quadratic space $V$ has signatures one of:
\[ (7,1),(3,5) \]
as these are precisely those with trivial Clifford invariants and discriminant $-1$.
Note that we cannot simply replace $Q$ by $-Q$ because of the definition of $\twistedcomp$.
\end{itemize}
\end{ex}

\subsection{Trialitarian Algebras}

Much of the material of this section comes from \cite[Ch. 43]{book_of_involutions}. A more thorough treatment of it can be found there.

\begin{df}
Let $L$ be a cubic \'etale algebra over $k$. Denote by $\Delta_L$ the discriminant algebra of $L$.
Let $D$ be a central simple $L$-algebra with orthogonal involution $\tau$.

We say that $(L,D,\tau,\alpha)$ is a \emphh{trialitarian algebra} if 
\[ \alpha : (\Clif_{D,\tau}^+,\tau)  \rightarrow^\rho (D\otimes_k\Delta_L, \tau\otimes 1) \]
where $\rho$ is an order three element of $\Gal(L\otimes_k\Delta_L/ \Delta_L)$ and $^\rho$ means that the map is $\rho$-skew linear, that is the action of $L$ on $(D\otimes_k\Delta_L, \tau\otimes 1 )$ is via the map $L\rightarrow L\otimes_k\Delta_L$ given by $\ell \mapsto \rho(\ell\otimes 1)$.
\end{df}

\begin{rmk}
Giving an isomorphism 
\[ \alpha : (\Clif_{D,\tau}^+,\tau)  \rightarrow^\rho  (D\otimes_k\Delta_L, \tau\otimes 1) \]
   is equivalent to giving an isomorphism:
\[ \tilde\alpha : (D,\tau) \oplus (\Clif_{D,\tau}^+,\tau)  \rightarrow (D\otimes_k L, \tau\otimes 1) \]
where $\tilde\alpha$ restricts to the first factor by the natural inclusion.

This observation together with Proposition \ref{prop:uniquealpha} and an argument as in Corollary \ref{cor:trivinvars} implies  that the discriminant of the involution $\tau$ is $\delta_L$ the discriminant of $L$.
\end{rmk}

\begin{thm}\label{thm:tritripclass}
Let $(L,D,\tau,\alpha)$ be a trialitarian algebra, then $[D]$ is in the kernel of the corestriction map from the Brauer group over $L$ to the Brauer group over $k$.
Moreover, over a number field, this is sufficient, that is any $[D]$ in the kernel of the corestriction map can be given the structure of a trialitarian algebra.
\end{thm}
\cite[43.6 and 43.8]{book_of_involutions}
\begin{rmk}
One may check conditions on the existence of an isomorphism locally, in this context we are reduced to the cases which arise in Remark \ref{rmk:CliffordStructure}.
It is worth noting that 
\end{rmk}

\begin{ex}\label{ex:trialg}
Let $(L,V,Q,\twistedcomp)$ be a twisted composition, then there exists a map 
\[ \alpha:(\Clif_{D,\tau}^+,\tau)  \rightarrow ^\rho(D\otimes_k\Delta_L, \tau\otimes 1)\]  such that the datum $(L,\End_L(V),\Ad_Q,\alpha)$, is a trialitarian algebra
\cite[36.19]{book_of_involutions}.
\end{ex}

\begin{prop}\label{prop:uniquealpha}
Suppose that $L$ is not a field, so that $L = k\oplus \Delta_L$. Then every trialitarian algebra over $L$ is of the form $(L,A \oplus C^+_{A,\tau},\tau\oplus \tau, \alpha)$
where $(A,\tau) \otimes \Delta_L \simeq  C^+_{A,\tau}$.
Moreover, the choice of $\alpha$ is unique up to unique isomorphism once a generator $\rho$ of $\Aut(L\otimes\Delta_L/\Delta_L)$ is chosen.
\end{prop}
See \cite[43.15]{book_of_involutions}.

\begin{cor}\label{cor:uniquealphaglobal}
Suppose $k$ is a global field and that $L$ is a degree three field extension.
Suppose also that $[D]$ is in the kernel of the corestriction map from the Brauer group over $L$ to the Brauer group over $k$ and that $(D,\tau_1)$ and   $(D,\tau_2)$ admit trialitarian structures $(L,D,\tau_1,\alpha_1)$ and $(L,D,\tau_2,\alpha_2)$.
Then $(D,\tau_1) \simeq (D,\tau_2)$ if and only this is true for all real places $\nu$ of $k$.

Suppose further that $\tau = \tau_1 = \tau_2$, then, up to the choice of generator $\rho$, there is an automorphism of $(L,D,\tau)$ taking $(L,D,\tau,\alpha_1)$ to $(L,D,\tau,\alpha_2)$.
\end{cor}
\begin{proof}
From the proposition we have:
\[ (D,\tau_1) \otimes L \simeq (D,\tau_1) \oplus (C^+_{D,\tau_1},\tau_1) \text{ and } (D,\tau_2) \otimes L \simeq (D,\tau_2) \oplus (C^+_{D,\tau_2},\tau_2).  \]
In particular:
\[ D \oplus C^+_{D,\tau_1} \simeq D\otimes L \simeq D \oplus C^+_{D,\tau_2}. \]
It follows that $\tau_1$ and $\tau_2$ have the same discriminants and Clifford invariants.
The first claim then follows by the classification of algebras with involutions over global fields.

The second claim is clear as $\alpha_1\circ\alpha_2^{-1}$ is an automorphism of $(L,D,\tau)\otimes \Delta_L$ and if there is an action on $\Delta_L$, its effect is to interchange $\rho$ with $\overline{\rho}$. 
\end{proof}
\begin{rmk}
The above result can be compared with \cite[43.11 - 43.14]{book_of_involutions} where the construction depends only on the choice of an octonion algebra.
\end{rmk}

\begin{thm}
Let $(L,D,\tau,\alpha)$ be the trialitarian algebra associated to a twisted composition $(L,V,Q,\twistedcomp)$ of dimension $8$, then
$\Aut(L,D,\tau,\alpha) \simeq \Aut(L,V,Q,\twistedcomp)^{\rm adj}$.

Moreover, trialitarian algebras are classified by $H^1(\Gal(\overline{k}/k), \Aut(L,D,\tau,\alpha))$.
\end{thm}
\cite[44.2 and 44.5]{book_of_involutions}

\begin{thm}
Let $(L,D,\tau,\alpha)$ be a trialitarian algebra over $k$, the connected component of the identity $\Aut(L,D,\tau,\alpha)^0$ is the subgroup which acts trivially on $L$.
\end{thm}
It suffices to check over the algebraic closure, see \cite[44.A]{book_of_involutions}.

\subsection{Classification of Groups of Type $D_4$}

\begin{cons*}
Let $(L,D,\tau,\alpha)$ be a trialitarian algebra over $k$.
Let $C$ denote the canonical map from $\Aut_L(D,\tau)$ to $\Aut_L(\Clif_{D,\tau}^+)$.
Let $\chi$ denote the canonical map from $\Spin(D,\tau)$ to $D$.
Then we may construct the following groups:
\begin{align*}
\Aut(L,D)^0(R) &= \{ g \in \Res_{L/k}( \Aut_L(D,\tau) )(R) \mid \alpha\circ C(g) = (C\otimes 1)\circ\alpha \} \\
 \Spin_{(L,D,\tau,\alpha)}(R) &= \{ g \in \Res_{L/k}(\Spin_{(D,\tau)/L})(R) \mid \alpha(g) = \chi(g)\otimes 1 \} 
\end{align*}
They are respectively, adjoint, and simply connected groups of type $D_4$.
 \end{cons*}

\begin{thm}
The adjoint groups (respectively simply connected groups) of type $D_4$ are in bijection with the trialitarian algebras over $k$,
both are classified by $H^1(\Gal(\overline{k}/k),\Aut(L,D,\tau,\alpha))$ the bijection associates to a trialitarian algebra, the groups defined above.
\end{thm}
\cite[44.8]{book_of_involutions}.

We briefly sketch the cohomological interpretation of the above, see \cite[Ch. 44]{book_of_involutions} for more details.

We have an exact sequence:
\[ H^1(\Gal(\overline{k}/k), \Aut(L,D)^0) \rightarrow H^1(\Gal(\overline{k}/k), \Aut(L,D)) \rightarrow H^1(\Gal(\overline{k}/k), S_3^\xi) \]
The group $\Aut(L,D)$ is both the automorphism group of the trialitarian algebra as well as the automorphism group of $\Aut(L,D)^0$.
Hence, $H^1(\Gal(\overline{k}/k), \Aut(L,D))$ classifies both.
In the above exact sequence the map to $H^1(\Gal(\overline{k}/k), S_3^\xi)$ is associating an \'etale algebra of degree $3$, in particular, the algebra $L$.
Note that in considering the classification one typically would use as base point the split form of the group in which case $S_3^\xi \simeq S_3$.
In the case where the group is split, $\Aut(\Gal(\overline{k}/L),D)^0$ is the automorphism group of the matrix algebra with the transpose involution preserving the determinant, and 
$H^1(\Gal(\overline{k}/k), \Aut(L,D)^0)$ classifies central simple algebras with orthogonal involution and trivial discriminant.
The first map thus is associating to a class of $H^1(\Gal(\overline{k}/k), \Aut(L,D)^0)$ a trialitarian algebra over $L$ constructed as in Example \ref{ex:split_twisted} and \ref{ex:trialg}.

Moreover, the natural inclusion:
\[  \Aut(L,D)^0 \injects \Res_{L/k}( \Aut_L(D,\tau) ) \]
Results in maps:
\[ H^1(\Gal(\overline{k}/k), \Aut(L,D)^0) \rightarrow H^1(\Gal(\overline{k}/k), \Res_{L/k}( \Aut_L(D,\tau)^0 )) \rightarrow H^1(\Gal(\overline{k}/L),\Aut_L(D,\tau)^0 ). \]
The Galois cohomology set $H^1(\Gal(\overline{k}/L),\Aut_L(D,\tau)^0 )$ classifies algebras with orthogonal involutions over $L$ of the same discriminant (this being equal to the discriminant of $L$ by Theorem \ref{thm:tritripclass}).
Further analysis reveals that the algebras  in the image of the map are also all in the kernel of the corestriction map to $k$, this follows directly from Theorem \ref{thm:tritripclass} but can also be shown by considering appropriate exact sequences, see \cite[Lem. 44.14]{book_of_involutions}.

\begin{rmk}\label{rmk:simpletri}
If $L/k$ is not a field, then $L = k \oplus F$, and $D$ has a correspond decomposition $D = D_k \oplus D_F$.
The spin group $\Spin_{(L,D,\tau,\alpha)} = \Aut(L,D,\tau,\alpha)^0$ is isomorpic to the spin group of $D_k$ and $D_F$ is simply the Clifford algebra of $D_k$.

If follows, that even if $L$ is not a field we still have:
 \[ \Spin_{(L,D,\tau,\alpha)} \otimes_k L \simeq \Spin_{D,\tau}. \]
\end{rmk}

\section{Maximal Tori in $D_n$}\label{sec:tori}

The main 
This section contains the major results of this paper, that is it describes the rational conjugacy classes of maximal tori in groups of type $D_n$.

We shall first need to give two auxiliary results on tori in specific situations.

\subsection{Conjugacy Classes of Tori in Covering Groups}

We first describe how to characterize the difference in conjugacy class between a group and a covering group.
We shall apply this in the cases $\PSO, \SO, \Spin$.

\begin{lemma}\label{lem:Covers}
Let $T$ be a torus in a connected semisimple group $G$, and let $\tilde{G}$ be a connected cover of $G$ and let $Z$ be the kernel of the map from $\tilde{G}$ to $G$. The rational conjugacy classes of tori $T'\subset \tilde{G}$ whose image in $G$ is a rational conjugates of $T$ are in bijection with:
\[ \im(H^0(\Gal(\overline{k}/k),G) \rightarrow H^1(\Gal(\overline{k}/k),Z)) / \im(H^0(\Gal(\overline{k}/k),T) \rightarrow H^1(\Gal(\overline{k}/k),Z). \]
\end{lemma}
\begin{proof}
In the following for brevity shall we omit the Galois group in when we write cohomology groups, ie: $H^\ast(X) = H^\ast(\Gal(\overline{k}/k),X)$.
Firstly, we recall that the rational conjugacy classes of tori in a group $G$ are in bijection with:
\[ \ker( H^1(N_G(T)) \rightarrow H^1(G)) \]
as arising from the exact sequence:
\[ H^0(G) \rightarrow H^0(G/N_G(T)) \rightarrow H^1(N_G(T)) \rightarrow H^1(G)). \]

In the case of the comparison between $\tilde{G}$ and $G$ we may consider the exact diagram:
\[ \xymatrix{ H^0(G)\ar@{->}[r] & H^1(Z) \ar@{->}[r]& H^1(\tilde{G})\ar@{->}[r] & H^1(G)  \\
     H^0(N_{G}(T)) \ar@{->}[r]\ar@{->}[u]& H^1(Z)\ar@{->}[r]\ar@{->}[u]^{\simeq} & H^1(N_{\tilde{G}}(T')) \ar@{->}[r]\ar@{->}[u]& H^1(N_{G})(T) \ar@{->}[u] }
\]
from which we deduce that we are interested in:
\[\im(H^0(G) \rightarrow H^1(Z) / \im(H^0(N_{G}(T)) \rightarrow H^1(Z)). \]
That is, the intersection of the kernels of the two maps out of $ H^1(N_{\tilde{G}}(T'))$ is precisely image of the top map into $H^1(Z)$ modulo the image of the bottom map into $H^1(Z)$.

Next, by considering the exact diagram:
\[ \xymatrix{ H^0( N_{G}(T)/T)\ar@{->}[r]& H^1(T)\ar@{->}[r] & H^1(N_{G}(T)) \ar@{->}[r]& H^1(N_{G}(T)/T) \\
    H^0( N_{\tilde{G}}(T')/T') \ar@{->}[r]\ar@{->}[u]^{\simeq}& H^1(T') \ar@{->}[r]\ar@{->}[u]& H^1(N_{\tilde{G}}(T'))\ar@{->}[r]\ar@{->}[u] & H^1(N_{\tilde{G}}(T')/T')\ar@{->}[u]^{\simeq} \\
 & H^1(\mu_2) \ar@{->}[u]\ar@{->}[r]^{\simeq} & H^1(\mu_2) \ar@{->}[u] & 
}
\]
we deduce that we have:
\[ \im(H^0(N_{G}(T)) \rightarrow H^1(Z)) \simeq \im(H^0(T) \rightarrow H^1(Z)).\]
This gives the desired result.
\end{proof}

\subsection{Tori in Restriction of Scalars Groups}

The following result is not strictly used in the sequel, but motivates some of the ideas we shall use when looking at tori in groups of type $D_4$.

\begin{lemma}\label{lem:ToriResScalar}
Let $G$ be a reductive group over $L$ a finite extension of $k$ an infinite field.
\[ H = \Res_{L/k}(G).\]
Let $T\subset H$ be a maximal $k$-torus of $H$, then:
\[ T = \Res_{L/k}(T') \]
 for $T'$ a maximal $L$-torus of $G$.
\end{lemma}
\begin{proof}
The collection of $k$-points of $T$ are all $k$-points of $H=G(L)$ and hence are $L$ points of $G$.
Set $T'$ to be the center of their centralizer in $G$, it is thus an $L$-torus of $G$.
Then, the Zariski closure of $T(k) \subset \Res_{L/k}(T')$, but $T(k)$ is dense in $T$, and thus $\Res_{L/k}(T') = T'$.
\end{proof}

\begin{rmk}
The result can be extended to finite fields by considering an appropriate base change and using Galois descent.
As we shall not use this result, we omit the details.
\end{rmk}

\subsection{Tori in Orthogonal Groups}\label{subsec:tori-clasic}

We shall first handle the case of orthogonal groups.
For orthogonal groups coming from quadratic spaces, this was handled in our previous work \cite{Fiori1}.

The key results are the following:

\begin{prop}\label{prop:prop_struct}
Let $q$ be a quadratic form over $k$ and let $\Orth_q$ be the associated orthogonal group. Let $T\subset \Orth_q$ be a maximal $k$-torus.
Then there exists an \'etale algebra with involution $(E,\sigma)$ over $k$ such that $T = T_{E,\sigma}$.
Moreover, suppose $T_{E,\sigma} \subset \Orth_q$ is a maximal torus. Then
\[ q(x) = \indnota{q_{E,\lambda}(x)} = \tfrac{1}{2}\Tr_{E/k}(\lambda x\sigma(x)) \]
for some choice of $\lambda\in (E^\sigma)^\ast$.

Given $q$ and $T_{E,\sigma}$ the choice of $\lambda$ is unique as an element of $(E^\sigma)^\times/N_{E/E^\sigma}(E^\times)\Aut_k(E,\sigma)$.
Moreover, two isomorphic tori $T_1 \simeq T_2 \simeq T_{E,\sigma}$ embedded into $\Orth_q$ with respect to $\lambda_1$ and $\lambda_2$ are $k$-conjugate if and only if $\lambda_1 = \lambda_2$ in $(E^\sigma)^\times/N_{E/E^\sigma}(E^\times)\Aut_k(E,\sigma)$.
\end{prop}
\begin{proof}
The first claim is \cite[Prop. 2.13]{Fiori1}, the second is interpreted Galois cohomologically in \cite[Ex. 6.126]{WalkerThesisTori}.

Now, given an embedding $T_{E,\sigma}$ into $\Orth_q$ the exact choice of $\lambda$ associated to the embedding depends only on a choice of isomorphism between the underlying vector space $V$ of $q$ and $E$ as $E$-modules.
Indeed, the form $q_{E,\lambda}(x)$ determines $\lambda$, thus so does a choice of form $q$ on $E$.
The choice in such an identification between $V$ and $E$ is only up to an $k$-automorphism of the $E$-module $V$, that is up to a scalar $y\in E^\times$ or a precomposition of the $E$-module structure map with an automorphism of $E$.
Changing this identification by $y\in E^\times$ modifies $\lambda$ by $N_{E/E^\sigma}(y)$, likewise modification by precomposition with a $k$-automorphism of $E$ acts naturally on $\lambda\in E$.
It follows, that associated to a triple $(q,T_{E,\sigma},\rho)$, where $\rho$ gives the embedding of $T_{E,\sigma}$ in $\Orth_q$, there is an invariant $\lambda\in (E^\sigma)^\times/N_{E/E^\sigma}(E^\times)\Aut_k(E)$.

Next, suppose $T_1,T_2\subset \Orth_q$ are both isomorphic to $T_{E,\sigma}$ then to give an element $g\in \Orth_q(k)$ conjugating them is equivalent to giving an isomorphism between the data $(q,T_{E,\sigma},\rho_1)$ and $(q,T_{E,\sigma},\rho_2)$ where $\rho_i$ gives the embedding of $T_i$ in $\Orth_q$. Two conjugate $T$ thus give isomorphic triples, which thus have the same invariant $\lambda$.

Conversely, consider two triples $(q,T_{E,\sigma},\rho_1)$ and $(q,T_{E,\sigma},\rho_2)$ for which $\lambda_1 = N_{E/E^\sigma}(y)\varphi(\lambda_2)$ for $y\in E^\times$ and $\varphi \in \Aut_k(E,\sigma)$.
Then the change of basis for $E$ associated to the endomorphism $y$ of $E$ transforms $q_{E,\lambda_1}(x)$ to $q_{E,)\varphi(\lambda_2)}(x)$, since $q_{E,\lambda_1}(x) = q = q_{E,)\varphi(\lambda_2)}(x)$ then we have $y\in \Orth_q(k)$ realizes the isomorphism between the triples $(q,T_{E,\sigma},\rho_1)$ and $(q,T_{E,\sigma},\rho_2\circ\varphi)$. The isomorphisms with $(q,T_{E,\sigma},\rho_2)$ is then realized via $\varphi$.
\end{proof}

\begin{rmk}
The group $\Aut_k(E,\sigma)$ which appears is naturally identified with 
\[ W(k) = (N_G(T)/T)(k)  \]
the rational points of the Weyl group.
\end{rmk}

\begin{thm}
\label{thm:brus_res}
Let $(E,\sigma)$ be an \'etale algebra with involution over $k$ of dimension $2n$ and let $\lambda\in (E^{\sigma})^{\times}$. Then the invariants of $q_{E,\lambda}(x) = \tfrac{1}{2}\Tr_{E/k}(\lambda x\sigma(x))$ are:
\begin{enumerate}
\item $D(q_{E,\lambda}) = (-1)^n\delta_{E/k}$,
\item $H(q_{E,\lambda}) = H(q_{E,1})\cdot\Cor_{E^\sigma/k}(\lambda,\delta_{E/E^\sigma})$,
\item $W(q_{E,\lambda}) = W(q_{E,1})\cdot\Cor_{E^\sigma/k}(\lambda,\delta_{E/E^\sigma})$.

\item For a real infinite place $\nu$ of $k$ the quadratic form has signature $(n+\tfrac{r}{2}-\tfrac{s}{2},n-\tfrac{r}{2}+\tfrac{s}{2})_\nu$ where $s$ (respectively $r$) is the number of real embeddings $\rho\in \Hom_{k-alg}(E^\sigma,\bR)$ of $E^\sigma$ which are ramified in $E$ with $\rho(\lambda)>0$ (respectively $\rho(\lambda)<0$).
\end{enumerate}
\end{thm}
\cite[Thm. 3.3 and Lem. 5.2]{Fiori1}.

\begin{thm} \label{thm:the_result}
Let $\Orth_q$ be an orthogonal group over a number field $k$ defined by a quadratic form $q$ of dimension $2n$ or $2n+1$, and let $(E,\sigma)$ be an \'etale algebra over $k$ with an involution and of dimension $2n$. Then
$\Orth_q$ contains a torus of type $(E,\sigma)$ locally everywhere (but not necessarily globally) if and only if the following three conditions are satisfied:
\begin{enumerate}
\item $E^\phi$ splits the even Clifford algebra $W^{\textrm{orth}}(q)$ for all $\sigma$-types $\phi$ of $E$.

\item \label{num22} If $\dim(q)$ is even then $\delta_{E/k} = (-1)^nD(q)$.
  
\item Let $\nu$ be a real infinite place of $k$ and let $s$ be the number of homomorphisms from $E$ to $\bC$ over $\nu$ for which $\sigma$ corresponds to complex conjugation.
      The signature of $q$ is of the form $(n-\frac{s}{2}+2i,n+\frac{s}{2}-2i)_\nu$ if the dimension is even and either $(n-\frac{s}{2}+2i+1,n+\frac{s}{2}-2i)_\nu$ or $(n-\frac{s}{2}+2i,n+\frac{s}{2}-2i+1)_\nu$ if $\nu((-1)^nD(q)\delta_{E/k})$ is respectively positive or negative when the dimension is odd, where $0\leq i\leq \tfrac{s}{2}$.
\end{enumerate}
Moreover, for any $E$ satisfying condition (\ref{num22}) we have that $\sqrt{D(q)} \in E^\phi$ for every $\sigma$-type $\phi$ of $E$.
\end{thm}
\cite[Thm. 5.1]{Fiori1}

We now reinterpret these in the context of orthogonal groups attached to algebras with orthogonal involutions.
The statements here appear more complicated largely as a consequence of the fact that there is no one canonical orthogonal involution to use as a base point for comparison.
The following proposition enumerates a number of useful facts concerning the structure of tori in these groups.
\begin{prop}
Let $(A,\tau)$ be a central simple algebra over $k$ with a symplectic involution $\tau$.
\begin{enumerate}
\item Suppose $\delta$ in $A^\times$ is such that $\tau(\delta)=-\delta$ and 
let $T\subset \Orth_{A,\tau_\delta}$ be a maximal $k$-torus.
Then there exists an \'etale algebra with involution $(E,\sigma)$ over $k$ such that $T = T_{E,\sigma}$ and $T \subset\Orth_{A,\tau_\delta}$ arises from an inclusion  \[ (E,\sigma)\injects (A,\tau_\delta). \]

\item
Fix an orthogonal involution $\tau'$ on $A$, any inclusion
\[ \rho: (E,\sigma) \injects (A,\tau') \]
induces an inclusion of the torus $ T_{E,\sigma} \injects \Orth_{A,\tau'}$.

\item 
Fix $\rho:(E,\sigma) \injects (A,\tau)$ for each $\delta \in E^\times$ such that $\sigma(\delta)=-\delta$ we have an inclusion:
\[ (E,\sigma) \injects (A,\tau_{\rho(\delta)}). \]

\item 
Fix an embedding $\rho:(E,\sigma)\injects (A,\tau)$.
The set of orthogonal involutions $\tau'$ on $A$ for which:
\[ \rho:(E,\sigma)\injects (A,\tau') \]
is precisely the set of involutions $\tau_{\rho(\delta)}$ for $\delta\in E^\times$ with $\sigma(\delta) = - \delta$.

\item
Fix an orthogonal involution $\tau'$ on $A$, and any inclusion $\rho':(E,\sigma) \injects (A,\tau')$.

Then there exists $g\in A^\times$ such that $\Int_g\circ\rho' : (E,\sigma) \injects (A,\tau)$, and 
\[ \Int_g\circ\rho' : (E,\sigma) \injects (A,\Int_g\circ \tau' \circ \Int_{g^{-1}})  \]
In particular:
\[ (A,\tau') \simeq (A,\Int_g\circ \tau' \Int_{g^{-1}}) = (A,\tau_{\Int_g\circ\rho'(\delta)}). \]
\item
Fix $\rho: (E,\sigma) \injects (A,\tau)$.
Consider any $\rho' : (E,\sigma) \injects (A,\tau)$ then there exists $g\in A^\times$ such that:
\[ \rho' = \Int_g \circ \rho \]
In particular for any collection:
\[ (E,\sigma,A,\tau',\rho') \]
(where $\tau'$ and $\rho'$ are what may vary)
it is equivalent to one of the form:
\[ (E,\sigma,A,\tau',\rho') \sim (E,\sigma,A,\tau'',\rho) \]
(where only $\tau''$ may vary). 
\item
Fix $\rho: (E,\sigma) \injects (A,\tau)$.
The isomorphism classes of collections:
\[ (E,\sigma,A,\tau',\rho') \]
(that is, varying $\tau'$ and $\rho'$)
are in bijection with elements $\delta\in E^\times$ such that $\sigma(\delta)=-\delta$ up to equivalence.

Two elements $\delta$ and $\delta'$ as above are equivalent if $\tau_{\delta} \simeq \tau_{\delta'}$ this equivalence is
 up to rescaling by $k^\times$ and the action of $E^\times$ on $A$ through $\Int_e$, that is $\delta \sim e\sigma(e)\delta$.

Moreover, fixing $\tau'$, the isomorphism classes of collections:
\[ (E,\sigma,A,\tau',\rho') \]
(that is, varying $\rho'$) 
are in bijection with elements $\delta\in E^\times$ such that $\tau_{\delta} \simeq \tau'$ up to rescaling by $k^\times$ and the action of $E^\times$ on $A$ through $\Int_e$, that is $\delta \sim e\sigma(e)\delta$.
\end{enumerate}
\end{prop}
\begin{proof}~
\begin{itemize}
\item
The first claim follows from Galois descent once we check over $\overline{k}$, where we may invoke Proposition \ref{prop:prop_struct}.
\item
The second and third claims are essentially immediate.
\item
For the forth claim, we note that all of the orthogonal involutions are of the form $\tau_{\epsilon}$ for $\epsilon \in A^\times$ with $\tau(\epsilon)=-\epsilon$.
The requirement that
 $\rho:(E,\sigma)\injects (A,\tau_\epsilon)$ given that $\rho:(E,\sigma)\injects (A,\tau)$ is that $\epsilon$ centralize the image of $\rho$.
We may check over the algebraic closure that $E$ is its own centralizer in $A$.
\item 
For the fifth claim we observe that for any $\delta'\in E^\times$ with $\sigma(\delta')= -\delta'$ we have:
\[ \tau'_{\delta'} \simeq \tau \]
and that $\tau'_\delta = \Int_{g^{-1}} \tau \circ \Int_{g}$ for some $g\in A$.

As we have $\rho':(E,\sigma) \injects (A,\tau'_\delta) = (A,\Int_{g^{-1}} \tau \circ \Int_{g})$ we immediately conclude:
\[ \Int_g\circ \rho'  :(E,\sigma) \injects (A,\tau) \]
and
\[ \Int_g\circ \rho'  :(E,\sigma) \injects (A,\Int_g \tau' \circ \Int_{g^{-1}}).  \]
It follows that 
\[(A,\Int_g \tau' \circ \Int_{g^{-1}}) = (A,\tau_{\Int_g\rho'(\delta)}). \]
\item
For the sixth claim it is an easy check that at least over $\overline{k}$ we may find:
\[ \tilde{g} \in A\otimes_k \overline{k} \]
such that $\Int_{\tilde{g}} \circ \rho = \rho'$. 

Consider any $\varphi \in \Gal(\overline{k}/k)$ then as $ g \rho(x) g^{-1} = \rho'(x)$ we have 
\[ \varphi( g \rho(x) g^{-1} ) = \varphi(\rho'(x)) = \rho'(x) =  g \rho(x) g^{-1}.\]
 It follows that $\tilde{g}\varphi(\tilde{g}^{-1})$ centralizes $\rho(E)$ and thus $ \tilde{g}\varphi(\tilde{g}^{-1}) \in E^\times$.

By Hilbert's Theorem 90 there exists $e\in E^\times$ with $\tilde{g}\varphi(\tilde{g}^{-1}) = {e}\varphi({e}^{-1})$ for all $\varphi \in \Gal(\overline{k}/k)$.
We then have that $g=\tilde{g}\rho(e^{-1}) \in A^\times$ satisfies:
 \[ \Int_{g} \circ \rho = \rho' .\]
\item
For the final claim we notice that by the above any datum:
\[ (E,\sigma,A,\tau',\rho') \]
is equivalent to one of the form:
\[ (E,\sigma,A,\Int_g\circ\tau'\circ\Int_{g^{-1}},\rho) \]
As we have fixed  $\rho: (E,\sigma) \injects (A,\tau)$ it further follows that $\Int_g\circ\tau'\circ\Int_{g^{-1}} = \tau_\delta$.

The ambiguity in the choice of $\tau_\delta$ which would in general be up to the action of $A^\times$, is now up to the elements of $A^\times$ which preserve $\rho$.
In particular the elements of $A^\times$ which centralize $\rho(E)$, that is, the elements of $E^\times$ acting by inner automorphism on $A$ through $\rho$.

The action of $e\in E^\times$ on $\tau_\delta$ is $\tau_{e\sigma(e)\delta}$ as $\tau(e) = \sigma(e)$ and $\delta$ commutes with $\sigma(e)$.\qedhere
\end{itemize}
\end{proof}

In the setting of endomorphism algebra's we had a natural base point, that is a choice of $\tau'$ from which we could build our classification. There is no canonical orthogonal involution, this is why in the above we continuously made use of symplectic involutions. In combination with the above, the following justifies why if we are contemplating any tori associated to $(E,\sigma)$ in $(A,\tau')$ there will exist a symplectic involution to use as a base point.
\begin{prop}\label{prop:existsymp}
Let $(A,\tau)$ be a central simple algebra over $k$ with a symplectic involution $\tau$.
Let $(E,\sigma)$ be an \'etale algebra with involution over $k$.
Then $(E,\sigma) \injects (A,\tau)$ if and only if $E\injects A$.
\end{prop}
\cite[Prop. 5.7]{CKMFrobeniusAlgebras}

When dealing with endomorphism algebras we used $\lambda\in (E^\sigma)^\times$, the following corollary, proposition and theorem rephrase the above in terms of this normalization.

\begin{cor}
Fix $(A,E,\sigma)$ with $E\injects A$.
The isomorphism classes of collections:
 \[ (A,\tau',E,\sigma,\rho) \]
is in bijection with elements $\lambda \in (E^\sigma)^\times$ modulo $k^\times \cdot N_{E/E^\sigma}(E^\times)$.
\end{cor}
\begin{proof}
By fixing any one $\delta$ with $\sigma(\delta) = -\delta$ all other such are of the form $\lambda\delta$ for $\lambda$ as above. The results is then immediate.
\end{proof}

\begin{prop}
With notation as above
Fix $\delta \in E^\times$ with $\sigma(\delta) = -\delta$. Then
\[ (A,\tau_{\lambda_1\delta},E,\sigma,\rho) \qquad (A,\tau_{\lambda_2\delta},E,\sigma,\rho) \]
give rationally conjugate tori in 
\[ \Orth_{A,\tau_{\lambda_1\delta}} \simeq \Orth_{A,\tau_{\lambda_2\delta}} \]
if and only if $\tau_{\lambda_1\delta} \simeq \tau_{\lambda_2\delta}$ and there exists $\varphi$ in $\Aut(E,\sigma)$ with $\tau_{\varphi(\lambda_1)\delta} = \tau_{\lambda_2\delta}$.
\end{prop}
\begin{proof}
The distinction between the tori and the embedded algebras is the insistence on identifying $\rho$ up to equality rather than simply up to having equal image.
It is thus clear by precomposing $\rho$ with $\varphi$ in $\Aut(E,\sigma)$ we obtain identified tori.

The procedure we used to identify the invariants $\lambda\delta$ required renormalizing relative to a fixed $\rho$.
That is, given the datum
\[ (A,\tau_{\lambda_1\delta},E,\sigma,\rho)  \sim (A,\tau_{\lambda_1\delta},E,\sigma,\rho\circ \varphi) \]
there must be $g\in A^\times$ such that $\rho\circ \varphi = \Int_g\circ \rho$.
By checking over the algebraic closure we can actually find $\tilde{g}\in \Orth_{A,\tau_{\delta}}$ to accomplish this.
We can then compute that action of $g = \tilde{g}e$ on $\tau_{\lambda_1\delta}$ is by the conjugation action on $\lambda_1$, which agrees with $\varphi$.
\end{proof}

\begin{rmk}
We could just as well have considered the action on $\lambda_2\delta$ to be $\varphi(\lambda_2\delta)$ by taking $g\in \Sp_{A,\tau}$.

Indeed there is a bijection between the two collections, the normalization with respect to $\lambda$ we have given is natural with respect to our previous work. It is not strictly clear if the alternative normalization might be more natural in general.
\end{rmk}

We now wish to upgrade the result from a $k$-isomorphism classes of embeddings of tori to a result about $k$-conjugacy classes of tori in $\SO_{A,\tau}$. 
The distinction between $k$-conjugacy classes of tori in $\SO_{A,\tau}$ and $k$-isomorphism classes of embeddings is characterized by conjugations not possible in $\SO_{A,\tau}$ but that are possible in $\PSO_{A,\tau}$, modulo conjugations by $T(k)$. This is precisely the result of Lemma \ref{lem:Covers}.
The elements $(E^\sigma)^\times/N_{E/E^\sigma}(E^\times)k^\times$ capture the conjugacy classes in $\PSO_{A,\tau}$. Our goal now is to explain why $(E^\sigma)^\times/N_{E/E^\sigma}(E^\times)$ characterizes conjugacy classes in $\SO_{A,\tau}$ (just as in the case of classical quadratic spaces).
In order to make sense of this we will need a concrete description of:
\[ H^1(\Gal(\overline{k}/k),\SO_{A,\tau}) \]
so as to normalize elements of $\lambda$ as we do in the classical case with respect to an explicit quadratic form.

Recall (from Section \ref{subsec:cohominterp} that we have a bijection:
\[ H^1(\Gal(\overline{k}/k),\SO_{A,\tau}) \simeq \{ (s,z) \in A^\times\times k^\times \;\mid\; \tau(s)=s\text{ and } N_{A/k}(s) = z^2 \}/\sim \]
where the equivalence relation on the right is given by $(s',z') \sim (s,z)$ if there exists $a\in A^\times$ with $s'= as\tau(a)$ and $z' = N_{A/k}(a)z$.
Recall also that the group $H^1(\Gal(\overline{k}/k),\mu_2) \simeq k^\times/ (k^\times)^2$ acts on the elements by $(s,z) \mapsto (fs,f^{n/2}z)$. 

We our now in a position to prove the following:
\begin{prop}\label{prop:upgradePSOtoSO}
Fix both an algebra with orthogonal involution $(A,\tau)$ as well as an element 
\[ (s,z) \in H^1(\Gal(\overline{k}/k),\SO_{A,\tau}) \]
which we use to determine another involution $\tau_s$ of $A$.
To any inclusion 
\[ (E,\sigma) \injects (A,\tau_s) \]
we can associate a collection of elements of the form $k^\times\lambda \in (E^\sigma)^\times$ for some $\lambda$ where $(\lambda, N_{E^\sigma/k}(\lambda))$ is equivalent to $(s,z)$.
The subset of the image of $k^\times\lambda$ in $E^\sigma/N_{E/E^\sigma}(E^\times)$ for which $(d\lambda, d^nN_{E/k}(\lambda)) \sim (\lambda, N_{E^\sigma/k}(\lambda))$ are in bijection with the conjugacy classes of $T_{E,\sigma}$ which become conjugate in $\PSO_{A,\tau}$.

In particular, the rational conjugacy classes of $T_{E,\sigma}$ in $\SO_{A,\tau}$ are in bijection with elements $\lambda\in (E^\sigma)^\times$ for which $(\lambda, N_{E^\sigma/k}(\lambda))$ is equivalent to $(s,z)$.
\end{prop}
\begin{proof}
We already know there is an element $\tilde{\lambda} \in (E^\sigma)^\times/N_{E/E^\sigma}(E^\times)k^\times$ which gives a form $\tau_{\tilde{\lambda}}$ equivalent to $\tau_s$.
Because $\tau_s$ and $\tau_{\lambda\delta}$ are isomorphic the image of $(s,z)$ and $(\lambda, N_{E^\sigma/k}(\lambda))$ coincide in $H^1(\Gal(\overline{k}/k),\Aut_{A,\tau})$.
They thus differ by the action of something in $H^1(\Gal(\overline{k}/k),\mu_2)$, which acts by scaling.
Thus, there is a scalar multiple of $\lambda of \tilde{\lambda}$ which is equivalent.

By Lemma \ref{lem:Covers} the distinct conjugacy classes in $\SO_{A,\tau}$ which collapse in $\PSO_{A,\tau}$ are in bijection with:
\[ \im(H^0(\Gal(\overline{k}/k),\PSO_{A,\tau}) \rightarrow H^1(\Gal(\overline{k}/k),\mu_2)) / \ker(H^1(\Gal(\overline{k}/k),\mu_2) \rightarrow H^1(\Gal(\overline{k}/k),T)). \]
This is precisely the subset of $k^\times$ for which $(ds,d^nz) \sim (s,z)$ modulo norms from $E^\sigma$.
\end{proof}

\begin{thm}\label{thm:lambdafortoriorth}
Fix $(A,\tau)$ a central simple algebra over $k$ with a symplectic involution $\tau$.
Let $(E,\sigma)$ be an \'etale algebra with involution over $k$, and fix $\delta = \sqrt{\delta_{E/E^\sigma}}$ so that $E\simeq E^\sigma(\delta)$.
Fix an inclusion $(E,\sigma)\injects (A,\tau)$.
Suppose $T_{E,\sigma} \subset \Orth_{A,\tau'}$ as a maximal $k$-torus.
Then there exists $\lambda\in (E^\sigma)^\times$ for which 
$(A,\tau') \simeq (A,\tau_{\lambda\delta})$ and the inclusion $T_{E,\sigma}\subset \Orth_{A,\tau'}$ arises from the fixed inclusion $(E,\sigma)\injects (A,\tau)$.

The rational conjugacy classes of the images of $T_{E,\sigma}$ in $\SO_{A,\tau'}$ are in bijection with the elements of
\[  \lambda \in E^\times /N_{E/E^\sigma}(E^\times)\Aut_k(E,\sigma) \]
for which $(\lambda,N_{E^\sigma/k}(\lambda)) \in H^1(\Gal(\overline{k}/k),\SO_{A,\tau'})$ is equivalent to a fixed one chosen to define $\tau'$.
\end{thm}
\begin{proof}
The existence of $\lambda$ follows from the above.

As in Proposition \ref{prop:upgradePSOtoSO} we have that the association $(A,\tau',E,\sigma,\rho) \mapsto \lambda$ is well defined up to $N_{E/E^\sigma}(E^\times)k^\times$, however, the association
$(T \subset \Orth_{A,\tau'}) \mapsto (A,\tau',E,\sigma,\rho)$ is only well defined up to $\Aut_k(E,\sigma)$. In particular, the choice of $\rho$ can be precomposed with any such automorphism.

It follows that having fixed a base point $(E,\sigma)\injects (A,\tau)$ we may associate to any torus $T_{E,\sigma} \subset \Orth_{A,\tau'}$ an element $\lambda$ well defined up to $N_{E/E^\sigma}(E^\times)\Aut_k(E,\sigma)k^\times$.

Conversely, if $\lambda_1 = N_{E/E^\sigma}(y)\varphi(\lambda_2)$ then the data:
\[ (A,\tau_{\lambda_1},E,\sigma,\rho_1) \simeq (A,\tau_{\varphi(\lambda_2)},E,\sigma,\rho_1) \simeq  (A,\tau_{\lambda_2},E,\sigma,\rho_1\circ \varphi^{-1}) \]
So that tori associated to the first and last are conjugate in $\PSO(A,\tau)$.
\end{proof}

\begin{rmk}
Fixing $(A,\tau,E,\sigma)$ not all options for $\lambda$ are associated to a fixed $\tau_s$, indeed the choice of $\lambda$ determines the isomorphism class of $\tau_\lambda$, and different choices of $\lambda$ may result in non-isomorphic $(A,\tau_\lambda)$.
Moreover, just as in the more classical case of quadratic forms, and as is often the case when dealing with Galois cohomology one should interpret the collection of $\lambda$ as a torsor relative to a fixed base point. Notice that in Proposition \ref{prop:upgradePSOtoSO} we need to fix the information of $\tau$, a choice $s$, and a preliminary embedding of $(E,\sigma)$.

In the more classical case of quadratic forms, we have often fixed the base point $\Tr_{E^\sigma/E}(x\sigma(x))$, but it is not the only choice and it may actually be more natural to prefer a different choice such as one which has trivial Clifford invariant and maximal index (see for example \cite[Lem. 3.6]{Fiori1}).
\end{rmk}

Based on the above we have that the group $(E^\sigma)^\times/N_{E/E^\sigma}(E^\times)$ acts transitively on the set of collections
\[ (A,\tau',E,\sigma,\rho) \]
where $\tau'$ is an orthogonal involution. We wish to be able to describe the invariants of the resulting forms.

\begin{thm}\label{thm:invariants_upgraded}
Let $(A,\tau)$ be a central simple algebra of degree $2n$ over $k$ with a symplectic involution $\tau$, and $(E,\sigma)$ be an \'etale algebra with involution over $k$ such that $E\injects A$.

Fix a representative $\delta = \sqrt{\delta_{E/k}}$

The invariants of $\tau' = \tau_{\lambda\delta}$ are as follows:
\begin{itemize}
\item The discriminant of $\tau'$ is given by $D(\tau) = (-1)^n\delta_{E/k}$ where $\delta_{E/k}$ is the discriminant of $E$.
\item 
If $(A,\tau'',E,\sigma,\rho) = \lambda\circ (A,\tau',E,\sigma,\rho)$ then 
 \[ [\Clif_{A,\tau''}]  = [\Clif_{A,\tau'}] + \Res_{Z/k}\Cor_{E^\sigma/k}((\delta_{E/E^\sigma},\lambda)). \]
In particular:
 \[ [\Clif_{A,\tau_{\lambda\delta}}]  = [\Clif_{A,\tau_{\delta}}] + \Res_{Z/k}\Cor_{E^\sigma/k}((\delta_{E/E^\sigma},\lambda)). \]
In the above $Z= k(\sqrt{(-1)^nD(\tau)})$ is the center of the Clifford algebra.
\item Suppose $A$ is a matrix algebra over a quaternion algebra.
Let $r$ (respectively $s$) be the number of real places of $E^\sigma$ which ramify in $E$ where $\lambda > 0$ (respectively $\lambda < 0$).
    The index of $(A,\tau)$ is precisely:
\[   (n - \abs{r-s}/2)/\deg([A])  \]
\end{itemize}
\end{thm}
\begin{proof}
The first point follows immediately from the definition of the discriminant of an orthogonal involution, and the definition of the discriminant of an \'etale algebra.

The second point is precisely \cite[Prop. 5.3]{Brus_OrthTori}. 

For the third point, we need only consider the case $E = \bC^n$ and $A = \Mn(\bH)$ as a earlier work (Theorem \ref{thm:the_result}).
We note that the inclusion $(E,\sigma)\injects (A,\tau_{\lambda\delta})$ induces the structure of a quadratic module on $E$.
We claim the index of $ (A,\tau_{\lambda\delta})$ is the same as the index of this quadratic space, from which the result follows.
By comparing to the structure of the ``quadratic space" we can define on $\bH^n$ the result then follows by an explicit check.
\end{proof}

\begin{rmk}
We should point out that the expression $\Res_{Z/k}\Cor_{E^\sigma/k}((\delta_{E/E^\sigma},\lambda))$ depends only on $\lambda$ if the discriminant $D(\tau)$ is trivial and $\delta_{E/E^\sigma}$ is non-trivial. It is worth comparing this to the possibilities considered in Remark \ref{rmk:CliffordStructure} as well as contrasting this result with Theorem \ref{thm:brus_res} which gives the Witt invariant of the related quadratic form.
\end{rmk}

Before proceeding we note that the following theorem makes use of results from the next section, but that there are no circular dependencies in the proofs.
\begin{thm} \label{thm:the_result_upgraded}
Let $k$ be a global field.
Let $(A,\tau')$ be a central simple algebra with orthogonal involution over $k$ of degree $2n$.
Let $(E,\sigma)$ be an \'etale algebra over $k$ with an involution and of dimension $2n$. Then
$\Orth_{A,\tau'}$ contains a torus of type $(E,\sigma)$ locally everywhere (but not necessarily globally) if and only if the following four conditions are satisfied:
\begin{enumerate}
\item $E\injects A$.

\item $\delta_{E/k} = (-1)^nD(\tau)$, in particular the center of the even Clifford algebra has discriminant $\delta_{E/k}$.

\item $E^\phi$ as an algebra over $k(\sqrt{\delta_{E/k}})$ splits the even Clifford algebra $\Clif^+_{A,\tau}$ over its center for all $\sigma$-types $\phi$ of $E$.

\item Let $\nu$ be a real infinite place of $k$ and let $s$ be the number of homomorphisms from $E$ to $\bC$ over $\nu$ for which $\sigma$ corresponds to complex conjugation.
      The index of $\tau$ is of the form $n-\frac{s}{2}+2i$ where $0\leq i\leq \tfrac{s}{2}$. Note that if $A_\nu$ is a quaternion algebra then $s=2n$ from the first condition.
\end{enumerate}
\end{thm}
\begin{proof}
By Theorem \ref{thm:invariants_upgraded} the only thing that we must show is that the condition related to the reflex algebra is equivalent to the isomorphism class of the Clifford algebra depending on $\lambda$.

Now, we make note of Theorem \ref{thm:tori_cliff_upgraded}, whose proof does not depend on this result, that whenever $(E,\sigma)$ does embed, then $E^\phi$ will split $\Clif^+_{A,\tau}$ over its center. In particular, the conditions above are all necessary, we must show they are sufficient. 

Moreover, as this result reduces to Theorem \ref{thm:the_result} when $A$ is an endomorphism algebra, we need only treat the cases where $A$ is a matrix algebra over a quaternion algebra.

In particular, we must only show that whenever $(E,\sigma)$ does not embed in $(A,\tau')$ then $E^\phi$ do not all split the even Clifford algebra over the center.

We complete the result with a case by case analysis, first we will be working over a $p$-adic local field $k_\nu$. The key observation is that by Remark \ref{rmk:CliffordStructure} there are at most two $\tau$ to consider for each choice of $A$.
\begin{itemize}
\item If there is a unique orthogonal involution on $A$ of discriminant $\delta_{E/k}$ then the assumption $E\injects A$ implies (using Proposition \ref{prop:existsymp})  that $(E,\sigma)\injects (A,\tau')$ and thus we have nothing to show.
\item If there are two orthogonal involutions on $A$ of discriminant $\delta_{E/k}$ and the isomorphism class of $\Clif^+_{A,\tau}$ over its center depends on $\lambda$, then $(E,\sigma)$ embeds into both $(A,\tau_1)$ and $(A,\tau_2)$ and again we have nothing to show.
\item If there are two orthogonal involutions on $A$ of discriminant $\delta_{E/k}$ and the isomorphism class of $\Clif^+_{A,\tau}$ over its center does not depend on $\lambda$, then (by Theorem \ref{thm:tori_cliff_upgraded}) $E^\phi$ splits $\Clif^+_{A,\tau}$ for the choice of involution $\tau$ for which $(E,\sigma) \injects (A,\tau)$.
It remains only to show that it does not split it for the other possible isomorphism class.

By \cite[Lem. 5.5]{Fiori1} and \cite[Cor. 2.6]{Fiori1} we see that if the isomorphism class does not depend on $\lambda$ then $E^\Phi$ has a field factor $k$.
The options for the Clifford algebra are $A\times M$ where $[A]$ is $2$-torsion and $[M]$ is trivial over center $k\times k$ or $A\times A^{op}$ where $[A]$ is $4$-torsion over $k\times k$. Since $E^\Phi$ must have a field factor $k$, the Clifford algebra it split must be $A\times M$.

Noting that $E^\Phi$ can then not split $M\times A$ as the $k$ factor would now need to split $[A]$, which it does not, completes the result.
\end{itemize}

We now deal with the case of $k_\nu = \bR$. Most cases are already covered by our previous result (Theorem \ref{thm:the_result}) the only case we still need to consider is $E = \bC^n$, in which case, the isomorphism class of $\Clif^+_{A,\tau}$ depends on $\lambda$ whenever the discriminant provides for more than one possible isomorphism of $\tau$. The conditions about the index are automatic in this case.
\end{proof}

\begin{rmk}
We should point out that our proof makes direct use of the structure of the Brauer groups for local fields via the use of Remark \ref{rmk:CliffordStructure} to classify the  possible $\tau$ over local fields. 
\end{rmk}

The question of when a torus which embeds locally everywhere embeds globally is somewhat subtle, for a thorough treatment of this question one may look at \cite{PR_localglobal}, \cite{BayerTori}, or \cite{BayerExpositoryToriClassical}. We remark first that the local global obstruction is one which relates to the algebra $E$, and not the algebra $(A,\tau)$.
We further point out that the conditions all come from attempting to find a global $\lambda$ given that one can find local everywhere $\lambda_\nu$.
As such, the case of central simple algebras is no different from that of quadratic spaces.

We now summarize several sufficient conditions on when an \'etale algebra will satisfy the local global principal.
\begin{thm}\label{thm:localglobal}
Let $(E,\sigma) = \oplus_i (E_i,\sigma_i) \oplus (F\oplus F,\sigma)$ where the $(E_i,\sigma_i)$ are all fields.
Then if $(E,\sigma)$ satisfies any of the following conditions, then $T_{E,\sigma}$ embeds as a maximal torus in $\Orth_q$ globally, if and only it does locally everywhere.
\begin{itemize}
\item For each pair $i,j$ there exists $p$ a prime of $k$ and $\fp_i,\fp_j| p$ primes of $E_i,E_j$ such that $\fp_i,\fp_j$ are both not split respectively over $E_i^\sigma,E_j^\sigma$.

\item $E$ is a field.

\item The algebra $\oplus_i (E_i,\sigma_i)$ is a CM-algebra.

\item Let $\tilde{E} = \prod_i E_i$ be the compositum of the $E_i$, there exists $\tilde\sigma$ an involution on $\tilde{E}$ such that $\tilde{\sigma}|_{E_i} = \sigma_i$.
\end{itemize}
\end{thm}
The first condition is from \cite{PR_localglobal}, the second is an immediate consequence. The third condition is precisely \cite[Cor. 4.1.1]{BayerTori} and the fourth is simply replacing complex conjugation with $\tilde{\sigma}$ in the proof of \cite[Lem. 2.2.2 and Cor. 4.1.1]{BayerTori}.

\subsection{Tori in Spin Covers of Orthogonal Groups}

\begin{thm}\label{thm:tori_cliff_upgraded}
Let $(A,\tau)$ be a central simple algebra with orthogonal involution over $k$ and 
let $(E,\sigma)$ be an \'etale algebra with involution over $k$ such that $T_{E,\sigma} \injects \Orth_q$ as a maximal subtorus. 

Then
$E^\Phi$ embeds into $\Clif_{A,\tau}$ as a maximal \'etale algebra over $k(\sqrt{\delta_{E/k}})$ stable under the canonical involution of $\Clif_{A,\tau}$ . Moreover, the canonical involution restricts to $\sigma$ on $E^\Phi$.
\end{thm}
\begin{proof}
This is precisely \cite[Thm. 2.18]{Fiori1} translated to this context.

We shall obtain the result by way of Galois descent working over the algebraic closure.

The inclusion $E\injects A$ gives us an inclusion of the idempotents $e_\rho$ into $A\otimes_k \overline{k}$, we shall denote the image in $A\otimes_k \overline{k}$  of $e_\rho$ by $\delta_\rho$ to avoid confusing the algebra structures when we eventually consider $T(A\otimes_k \overline{k})$.

Working over the algebraic closure $\overline{k}$ we may identify $A\otimes_k \overline{k}$ with the endomorphism algebra of $E\otimes_k \overline{k}$ (as a $\overline{k}$-module).
Indeed, as $E$ gives a maximal \'etale subalgebra of $A$ any realization of $A$ as an endomorphism algebra of a vector space gives that vector space the structure of a rank $1$ module over $E\otimes_k \overline{k}$. Fix any isomorphism $\varphi: A\otimes_k \overline{k} \simeq \End_{\overline{k}}(E\otimes_k \overline{k})$.
By construction the inclusion of $E\injects \End_{\overline{k}}(E\otimes_k \overline{k})$ induced by the inclusion of $E\injects A$ and the map $\varphi$ agrees with the inclusion $E\injects \End_{\overline{k}}(E\otimes_k \overline{k})$ induced from left multiplication of $E$ on $E$.

We shall need the following:
\begin{lemma}\label{lem:torispincovers}
In \cite[Thm. 2.18]{Fiori1} the association of $\delta_\rho = \frac{1}{\rho(\lambda)}e_\rho\otimes e_{\rho\circ\sigma} \in \Clif_{A,\tau}^+$ is the same as the association defined above of $\delta_\rho = e_\rho$ under the inclusion $E\injects A \injects  \Clif_{A,\tau}^+$.
\end{lemma}
\begin{proof}
Indeed, as an element of $\End_{\overline{k}}(E\otimes_k \overline{k})$ the element $\frac{1}{\rho(\lambda)}e_\rho\otimes e_{\rho\circ\sigma}$ acts as left multiplication by $e_\rho$.
\end{proof}

Using this it follows that the formal properties of the $\delta_\rho$ as defined in \cite[Thm. 2.18]{Fiori1} hold for $\delta_\rho$ as defined above.
In particular we have that:
\begin{enumerate}
\item The action of $\sigma$ on $\delta_\rho$ agrees with the canonical involution of $\Clif_{A,\tau}^+$,
\item $\delta_\rho^2 = \delta_\rho$,
\item $\delta_\rho\sigma(\delta_\rho) = 0$ and $\delta_\rho + \sigma(\delta_\rho) = 1$,
\item the $\delta_\rho$ all commute, and
\item the Galois action on $\{\delta_\rho\}$ is the same as that on $\{e_\rho\}$.
\end{enumerate}
Now for each $\sigma$-type $\phi\in\Phi$ of $E$ set $\delta_\phi = \prod_{\rho\in\phi} \delta_\rho$. These elements then satisfy the following properties:
\begin{enumerate}
\item $\delta_\phi^2 = \delta_\phi$,
\item $\delta_{\phi_1}\delta_{\phi_2} = 0$ for $\phi_1\neq\phi_2$,
\item $\sum_\phi \delta_\phi = \prod_\rho (\delta_\rho+\delta_{\rho\circ\sigma}) = 1$, and
\item the Galois action on $\{\delta_\phi\}_{\phi\in\Phi}$ is the same as that on $\{\phi\}_{\phi\in\Phi}$.
\end{enumerate}
Thus the $\delta_\phi$ are Galois stable orthogonal idempotents and hence by taking Galois invariants give an \'etale subalgebra of $\Clif_{A,\tau}^+$. As the Galois action on idempotents matches that of $E^\Phi$, this gives an embedding of $E^\Phi$ into $\Clif_{A,\tau}^+$. Moreover, this algebra is preserved by the canonical involution of $\Clif_q$, and the involution restricts to $\sigma$ on it.

The algebra is maximal as an \'etale subalgebra for dimension reasons.

Finally, we observe that $E^\Phi$ has the subalgebra $k(\sqrt{\delta_{E/k}})$ defined as follows:
Fix any one reflex type $\phi$ let:
\[ \Phi_+ = \{ \phi' \mid \abs{\phi\cap\phi'} \text{ is even}\} \]
and 
\[ \Phi_- = \{ \phi' \mid \abs{\phi\cap\phi'} \text{ is odd}\} \]
Then $\{ \Phi_+,\Phi_-\}$ does not depend on $\phi$ (though the order does) and admits a natural action of the Galois group.
We may then define:
\[ \delta_{\Phi_\pm} = \sum_{\phi\in \Phi_{\pm}} \delta_\phi \]
Then the \'etale sub-algebra of both $E^\Phi$ and $\Clif_{A,\tau}$ defined by $\delta_{\Phi_\pm}$ is the center of $\Clif_{A,\tau}$ and isomorphic to $k(\sqrt{\delta_{E}})$.

In particular $E^\Phi$ embeds in $\Clif_{A,\tau}$ viewing both as algebras over $k(\sqrt{\delta_{E}})$.
\end{proof}

\begin{ex}
Let $k$ be a non-archimidean local field, Suppose $A$ a central simple algebra of degree $2n$ is not a matrix algebra, and $\tau$, an orthogonal involution on $A$ has $(-1)^nD(\tau) = 1$. 
Suppose that $E= F\times F$ with $\sigma$ acting to interchange factors so that $\delta_{E/E^\sigma}=1$ is trivial.
Suppose $E\injects A$.

In this setting there were $2$ choices for $\tau$, these choices are distinguished by the isomorphism class of $[\Clif_{A,\tau}^+]$ over its center (which is $k\times k$). We recall (Remark \ref{rmk:CliffordStructure}) that $[\Clif_{A,\tau}]$ is the direct sum of two non-isomorphic central simple algebras, hence one is a matrix algebra, the other is not.
Even though these algebras are isomorphic over $k$, the algebra $(E,\sigma)$ injects into only one of the two algebras as $\Res_{Z/k}\Cor_{E^\sigma/k}((\delta_{E/E^\sigma},\lambda)$ does not depend on $\lambda$.
What we have is that $(E^\Phi,\sigma)$, as a $k$-algebra, has a unique non-quadratic factor $k\times k$, on which $\sigma$ acts by interchanging factors.
When $(E,\sigma) \injects (A,\tau)$ then as an algebra over $k\times k$, the center of the Clifford algebra, this $k\times k$ factor of $E^\Phi$ sits over $k$-factor of $Z$ over which the matrix algebra also sits.
\end{ex}

\begin{thm}\label{thm:TORIINSPIN}
Let $(A,\tau)$ be a central simple algebra with orthogonal involution over $k$.

Let $T$ be a torus in $\Spin_{A,\tau}$, then there exists $(E,\sigma)$ an \'etale algebra with involution over $k$ such that $T_{E,\sigma} \injects \Orth_{A,\tau}$.
The map $\chi : \Spin_{A,\tau} \rightarrow  \Orth_{A,\tau}$ maps $T$ to $T_{E,\sigma}$.

The torus $T\injects T_{E^\Phi,\sigma}$ and is isomorphic to the natural image of the map:
\[ T_{E,\sigma} \overset{N^\Phi}\longrightarrow T_{E^\Phi,\sigma} \]
defined by sending $x\in E$ to its reflex norm in $E^\Phi$.
The composition of the maps $\chi\circ N^\Phi$ is the map $x\mapsto x^2$ of $T_{E,\sigma}$.

We may thus view $T$ as the subtorus of $T_{E^\Phi,\sigma} \times T_{E,\sigma}$ defined by either:
\[ T = \{ (x,y) \in T_{E^\Phi,\sigma} \times T_{E,\sigma} \mid y = \chi(x) \} =  \{ (x,y) \in T_{E^\Phi,\sigma} \times T_{E,\sigma} \mid N^\Phi(y) = x^2 \}^{0}. \]

\end{thm}
\begin{proof}
The check this result we need to explicitly understand the map $\chi:\Spin_{A,\tau} \rightarrow \SO_{A,\tau}$. We have avoided giving the explicit construction of this map.
However, all of the results can be checked over an algebraic closure of $k$.

In this setting we may use the standard construction of this map using the full Clifford algebra.
The full Clifford algebra is a quotient of $T(V)$ where $V$ is the underlying vector space for a quadratic form.
We note that in our setting we may write $T(V) = V \otimes T(\End(V)) =  E\otimes T(A)$ via the identification of the underlying vector space with $E$.
This gives us a new map $E\injects \Clif_{A,\tau}$. We note that this map is injective and as in the proof of the previous theorem it takes $e_\rho$ to $e_\rho$.
The spin group is contained in the subgroup of $\Clif_{A,\tau}$ for which $gEg^{-1} \subset E$ and it is this natural action on $E$ that induces the map $\chi$.

A direct computation as in the remark following \cite[Thm. 2.18]{Fiori1} or in more detail the proof of \cite[Lemma. 4.5]{Brus_OrthTori} allows us to compute the actions of $\delta_\rho$ on $e_{\rho'}$ and from this we may compute the action of $N^\Phi(x)$ on $E$. In particular, we find that $\chi\circ N^\Phi$  is the map $x\mapsto x^2$. The other claims are then a direct consequence.
\end{proof}

\begin{df}\label{df:psi}
We define a map $\Psi: E^\Phi \rightarrow E \otimes_k E^\Phi$ as follows:
\[ \Psi\left( \underset{\phi}\sum a_\phi e_\phi \right) = \underset{\phi_1\cap\phi_2 = \{\rho\}}\sum a_{\phi_1}a_{\phi_2} e_\rho\otimes (e_{\phi_1}+e_{\phi_2} + e_{\overline{\phi}_1}+e_{\overline{\phi}_2}). \]
First, we note that the image of this map actually lands in a subalgebra $\tilde{E}$ of $E \otimes_k E^\Phi$, in particular the image is stable under the involution $1\otimes \sigma$ and the involution which takes $e_{\rho}\otimes e_{\phi_1}$ to $e_\rho \otimes e_{\phi_2}$ where $\phi_2 \cap \phi_1 \subset \{\rho,\overline{\rho}\}$. We shall make some use of this later.
For now, the key feature of this map which we shall use is that $\Psi \circ N^\Phi:E^\times\rightarrow E\otimes E^\Phi$ is the map $x \mapsto (N(x)x\sigma(x)^{-1})\otimes 1$ so that in particular on restriction to $T_{E,\sigma}$ it is the map $x\mapsto x^2\otimes 1$.

It follows that in the above theorem we have the intrinsic description of $T$ as:
\[ T = \{ (x,y) \in T_{E^\Phi,\sigma} \times T_{E,\sigma} \mid y\otimes 1 = \Psi(x) \} \]
moreover $\chi$ and $\Psi$ essentially define the same map.
\end{df}

\begin{rmk}
The key to carrying out the above computation explicitly without passing to the algebraic closure is the observation that our maps arise from the inclusion $E\injects A \injects T(A)$. As the even Clifford algebra is a quotient of $T(A)$ it is this inclusion of $E$ into $\Clif_{A,\tau}^+$ which via inner automorphisms induces the map $C: \SO_{A,\tau} \rightarrow \Aut(\Clif_{A,\tau}^+)$ as in the definition of the spin group (Proposition \ref{prop:AtoCliff}).
It thus suffices to check that the image of $x\in T_{E,\sigma}$ through this inclusion agrees with the image of $N^{\Phi}(x) \in E^\Phi$.
To check this identification one uses the relations from the ideal $J_2$ as used to define the Clifford algebra in \cite[Def. 8.7]{book_of_involutions}.

An alternate approach, more in line with the direct computations referred to in the sketch above would use the Clifford bimodule \cite[Ch. 9]{book_of_involutions}.
\end{rmk}

We now consider the problem of understanding the rational conjugacy classes of these tori.
Because of Lemma \ref{lem:Covers} it is natural to look at the following:
\begin{lemma}\label{lem:CohomKernSpin}
Let $T=T_{E,\sigma}$ and consider $T'$ the unique torus which can cover it in a spin group and the exact sequence:
\[  1 \rightarrow {\mu_2} \rightarrow T' \rightarrow T \rightarrow 1. \]
In the associated long exact sequence we have that:
\[\im(H^0(\Gal(\overline{k}/k),T) \rightarrow H^1(\Gal(\overline{k}/k),\mu_2) \simeq N_{E/k}(E^\times)/(k^\times)^2 \subset k^{\times}/(k^\times)^2. \]
\end{lemma}
\begin{proof}
We consider the exact diagram:
\[ \xymatrix{ &&&  \mu_2\ar@{->}[dd] \\
                         & \Res_{E^\sigma/k}(\mu_2) \ar@{->}[rru]^{N_{E^\sigma/k}}\ar@{->}[dr]& & \\
                          H  \ar@{->}[ur]  \ar@{->}[rr]& & T \ar@{->}[r]^{N_{\Phi}}\ar@{->}[dr]^{x^2}& T' \ar@{->}[d]^{\Psi}\\
                          & & & T }
\]
from which we deduce that 
\[ \im(H^0(\Gal(\overline{k}/k),T) \overset{\delta_{\Phi}}\longrightarrow H^1(\Gal(\overline{k}/k),\pm1) \]  is \[ N_{E^\sigma/k}(\im(H^0(\Gal(\overline{k}/k),T) \overset{\delta_{x^2}}\longrightarrow H^1(\Gal(\overline{k}/k),\Res_{F/k}(\mu_2))). \] 
We then recall that $\im(H^0(\Gal(\overline{k}/k),T) \rightarrow H^1(\Gal(\overline{k}/k),\Res_{E^\sigma/k}(\mu_2)))$ is precisely
 \[ \ker( H^1(\Gal(\overline{k}/k),\Res_{E^\sigma/k}(\mu_2))\rightarrow H^1(\Gal(\overline{k}/k),T))\]
 which is precisely $N_{E/E^\sigma}(E^\times)/((E^\sigma)^\times)^2$.

The result then follows immediately.
\end{proof}

\begin{rmk}
We can concretely interpret the coboundary map $T(k) \rightarrow H^1(\Gal(\overline{k}/k),\mu_2)$ as follows.
Every element $f\in T(k)$ is of the form $f = e/\sigma(e)$ for some $e\in E$. The image of $f$ in $H^1(\Gal(\overline{k}/k),\mu_2) / k^\times / (k^\times)^2$ is precisely $N(e)$.
\end{rmk}

\begin{thm}\label{thm:CONGTORISPINCOVERS}
Let $T=T_{E,\sigma}$ be a torus in $\SO_{A,\tau}$, the rational conjugacy classes of tori $T'\subset \Spin_{A,\tau}$ whose image in $\SO_{A,\tau}$ is a rational conjugates of $T$ are in bijection with a subset of:
\[ k^\times/N_{E/k}(E^\times)(k^\times)^2. \]
In particular, it is the subset generated by the images of $N_{E/k}(E^\times)$ in  $k^\times/(k^\times)^2$ as we vary over all pairs $(E,\sigma)$ which give tori in $\SO_{A,\tau}$.

The subgroup of $k^\times/(k^\times)^2$ which must be considered is precisely the image of the spinor norm, which is the value group of the quadratic form.
In the specific case of a local or global field $k$ this is:
\[ \{ x\in k^\times | \nu(x) > 0 \text{ for } \nu \text{ real place of } k \text{ where } \SO_{A,\tau} \text{ is compact}\}. \]
\end{thm}
\begin{proof}
The first claim follows immediately from the Lemmas \ref{lem:Covers} and \ref{lem:CohomKernSpin}.

The claim about the structure of the image follows from the fact that the coboundary map is a group homomorphism and orthogonal group is generated by semi-simple elements.
The final claim in the case of local or global fields follows by observing that the $H^1(\Spin_{A,\tau})$ is supported at the real places, and concretely studying the possible tori in this case.
\end{proof}

\begin{rmk}
To concretely interpret the above what it says is that if we conjugate a torus $T$ by an element whose spinor norm (its image in $H^1(\Gal(\overline{k}/k),\mu_2)$) is not in $(k^\times)^2$ then the preminage of that torus in in the spin group is not conjugate to the original unless there is an element $f\in T(k)$ with the same spinor norm.
As the spinor norms of elements $T(k)$ are in $N_{E/k}(E^\times)$ this explains the result.
\end{rmk}

\subsection{Tori in Simply Connected  Groups of type $D_4$}

\begin{lemma}
Let $(L,D,\tau,\alpha)$ be a trialitarian algebra.
Then every maximal torus:
\[ T \subset \Spin_{(L,D,\tau,\alpha)}(R) = \{ g \in \Res_{L/k}(\Spin_{(D,\tau)/L})(R) \mid \alpha(g) = \chi(g)\otimes 1 \}\]
 is contained in a unique maximal torus of $\Res_{L/k}(\Spin_{D,\tau})$.
\end{lemma}
\begin{proof}
The center of the centralizer of $T$ in $\Res_{L/k}(\Spin_{(D,\tau)/L})$ is a $k$-torus containing $T$.
To show that it is maximal it suffices to consider the base change to $\overline{k}$.

Since over $\overline{k}$ the group $\Res_{L/k}(\Spin_{(D,\tau)/L})$  is a product group we may consider the centralizer for each factor separately and compute it with respect to the projection onto each factor. This projection gives a maximal torus in each factor, we thus conclude that the centralizer is in each case a maximal torus.
\end{proof}

\begin{lemma}
Let $(L,D,\tau,\alpha)$ be a trialitarian algebra.
Let $T \subset \Spin_{(D,\tau)/L}$ be a maximal torus.
Let $(E,\sigma)$ be the \'etale algebra over $L$ associated to the torus $T_{E,\sigma}$ in $\SO_{D,\tau}$ whose preimage contains $T$ in $\Res_{L/k}(\Spin_{(D,\tau)/L})$.

Then $\alpha$ induces a $k$-algebra isomorphism $(E^\Phi,\sigma) \rightarrow (E \otimes_k \Delta_L,\sigma \otimes 1)$ which restricts to $\rho$ on the subalgebra $L\injects L\otimes_k \Delta_L$.
Consequently we have an isomorphism $\beta: (E,\sigma)\otimes_k L \overset\sim\rightarrow  (E\times E^\Phi,\sigma\times \sigma)$ given by:
\[ e\otimes \ell \mapsto (\ell e, \rho^{-1}(\ell)\alpha^{-1}(e)). \]
\end{lemma}
\begin{proof}
Recall that by definition $\alpha$ is a map \[ (\Clif_{D,\tau}^+,\tau) \rightarrow^\rho (D\otimes_k\Delta_L, \tau\otimes 1). \]

We have by Lemma \ref{lem:torispincovers} that $E^\Phi \injects \Clif_{D,\tau}^+$ and thus $\alpha|_{E^\Phi}$ gives a map $(E^\Phi,\sigma) \rightarrow^\rho (D\otimes_k\Delta_L, \tau\otimes 1)$.

Now, over the algebraic closure, the map $\alpha$ is essentially unique (see Proposition \ref{prop:uniquealpha}) and an explicit check in this case shows that the restriction of $\alpha$ to $\Spin_{D,\tau}$ is giving map:
\[ \Spin_{D,\tau} \hookrightarrow \SO_{D,\tau} \times \SO_{D,\tau} \]
where the decomposition on the right is with respect to $\Delta_L\otimes\overline{k} = \overline{k} \times \overline{k}$. Moreover, both projections to $\SO_{D,\tau}$ can be viewed as the map arising in the definition of $\Spin_{D,\tau}$ (see Proposition \ref{prop:spinAtau}). 
It follows then that 
 \[ \alpha(T(\overline{k})) \subset T_{E,\sigma}(\overline{k}). \]
and hence by descent that:
 \[ \alpha(T(R)) \subset T_{E,\sigma}(R) \]
for all $k$-algebras $R$.
Next, we by noting that $E$ is the $L$ linear span of $T_{E,\sigma}(L)$ and $E^\Phi$ is the $L$-linear span of $T(k)$, it follows by the $L^\rho$-linearity of $\alpha$ that:
\[ \alpha(E^\Phi) =E\otimes \Delta_L. \]

The final claim follows by exploiting that $L\otimes_k L \simeq L \times L\otimes_k\Delta_L$.
\end{proof}

\begin{rmk}
Via the isomorphism $(D,\tau)\otimes_k L = (D,\tau) \times (D,\tau)\otimes_L \Delta_L$ we obtain a projection onto $(D,\tau)$ and one onto $(D,\tau)\otimes_L \Delta_L$.
As the the base change of $\Spin_{(L,D,\tau,\alpha)/L}$ to $L$ preserves this decomposition we obtain maps from it to the groups preserving the other remaining structure. The projection onto the $(D,\tau)$ factor is giving the map from $\Spin_{(L,D,\tau,\alpha)/L}$ to $\SO_{(D,\tau)/L}$ whereas the projection onto $(D,\tau)\otimes_L \Delta_L$ is inducing the map $\Spin_{(L,D,\tau,\alpha)/L}$ to $\Spin_{(D,\tau)/L}$. This latter fact follows from the definition of $\beta$ and the definition of $\Spin_{(L,D,\tau,\alpha)/L}$ relative to $\alpha$.
The group $\Spin_{(D,\tau)/L}$ is precisely the double cover of $\SO_{(D,\tau)/L}$ given by $\chi$ and this projection is in fact giving $\chi$.

Note well that these maps only exist after base change to $L$ (or if $L$ is not a field).
\end{rmk}

\begin{lemma}\label{lem:structureofT}
Let $(L,D,\tau,\alpha)$ be a trialitarian algebra.
Let $T \subset \Spin_{(L,D,\tau,\alpha)}$ be a maximal torus and $(E,\sigma)$ the associated algebra as above and denote by $\beta$ the isomorphism:
\[\beta:  (E,\sigma)\otimes L \overset\sim\rightarrow (E\times E^\Phi,\sigma\times \sigma).\]

Then $T = T_{E,\sigma,\beta} \subset T_{E,\sigma}$ is the torus:
\[ T_{E,\sigma,\beta} (R) = \{ x \in (E\otimes_k R)^\times \mid  x\sigma(x) = 1,\; \beta(x) = (x_E,x_{E^\Phi}) \text{ satisfy } \Psi(x_{E^\Phi}) = x_E\otimes1 \}.  \]
Where $\Psi$ is the map 
\[ \Psi : E^\Phi \rightarrow E\otimes_k E^\Phi\]
defined in Definition \ref{df:psi}.
\end{lemma}
\begin{proof}
The content of this Lemma is the intrinsic conditions which define $T$.
The proof follows by inspecting the definition of $ \Spin_{(L,D,\tau,\alpha)}$ as
\[  \Spin_{(L,D,\tau,\alpha)}(R) = \{ g \in \Res_{L/k}(\Spin_{(D,\tau)/L})(R) \mid \alpha(g) = \chi(g)\otimes 1 \}  \]
and relating the maps $\alpha$ and $\chi$ to the maps $\beta$ and $\Psi$.

The connection between $\Psi$ and $\chi$ is discussed in Definition \ref{df:psi}.

The connection between $\beta$ and $\alpha$ is discussed in the proceeding remark.
\end{proof}

\begin{lemma}\label{lem:existenceofT}
Let $(L,D,\tau)$ be a datum for which there exists $\alpha'$ making it a trialitarian algebra.
Let $(E,\sigma)$ be the \'etale algebra over $L$ associated to a torus $T_{E,\sigma}$ in $\SO_{D,\tau}$.
Suppose further that we have a map:
\[ \beta : (E,\sigma)\otimes_k L \overset\sim\rightarrow (E\times E^\Phi,\sigma\times \sigma). \]
Then $\beta$ induces an $\rho$-isomorphism $\alpha: E^\Phi \rightarrow^\rho   (E,\sigma)\otimes_k \Delta_L$: which in turn extends uniquely to a map
\[ \tilde\alpha: (\Clif_{D,\tau}^+,\tau)  \rightarrow^\rho (D,\tau) \otimes_k \Delta_L  \]
making $(L,D,\tau,\tilde\alpha)$ into a trialitarian algebra for which $T_{E,\sigma,\beta}$ is a torus of the associated spin group.

The isomorphism class of $\tilde\alpha$, that is its rational conjugacy, is uniquely determined by$\beta$ (and $\rho$.)
\end{lemma}
\begin{proof}
First we recall that by \cite[Section 43.A]{book_of_involutions} the existence of the map:
\[ \beta : (E,\sigma)\otimes_k L \overset\sim\rightarrow (E\times E^\Phi,\sigma\times \sigma). \]
is equivalent to that of the map:
\[ \alpha : E^\Phi \rightarrow^\rho  (E,\sigma)\otimes_k \Delta_L. \]

Next we note that $(E^\Phi,\sigma)$ is a maximal \'etale subalgebra of $(\Clif_{D,\tau}^+,\tau)$ and $(E,\sigma)$ is a maximal \'etale subalgebra of $(D,\tau)$ so that $(E,\sigma)\times (E^\Phi,\sigma)$ is a maximal \'etale subalgebra of $(D,\tau) \times (\Clif_{D,\tau}^+,\tau)$ and $(E,\sigma)\otimes_k L$ is a maximal \'etale subalgebra of $(D,\tau) \otimes_k L$.
The map $\beta$ is thus giving an isomorphism between maximal \'etale subalgebras of $(D,\tau) \times (\Clif_{D,\tau}^+,\tau)$ and $(D,\tau) \otimes_k L$.

The assumption that a map $\alpha'$ exists implies $(D,\tau) \times (\Clif_{D,\tau}^+,\tau)$ and $(D,\tau) \otimes_k L$ are isomorphic.
The classification of maximal \'etale subalgebras with involutions of such algebra's with involutions implies then that $\beta$ is the restriction of a unique isomorphism $\alpha: (D,\tau) \times (\Clif_{D,\tau}^+,\tau) \rightarrow (D,\tau) \otimes_k L$.

The map $\tilde\alpha$ is then the corresponding map induced by $\alpha$.
\end{proof}

\begin{thm}\label{thm:bigone}
Let $(L,D,\tau,\alpha)$ be a trialitarian algebra.

The tori $T \subset \Spin_{(L,D,\tau,\alpha)}$ are of the form:
\[ T_{E,\sigma,\beta}(R) = \{ x \in (E\otimes_k R)^\times \mid  x\sigma(x) = 1,\; \tilde{\beta}(x) = (x_E,x_{E^\Phi}) \text{ satisfy } \Psi(x_{E^\Phi}) = x_E \otimes 1 \}  \]
for algebras $(E,\sigma)$ for which:
\begin{itemize}
\item $E$ is an \'etale algebra over $L$.
\item $(E,\sigma)$ is associated to a torus in $\SO_{(D,\tau)}$.
\item the extension of the isomorphism:
\[ \beta : (E,\sigma)\otimes_k L \overset\sim\rightarrow (E\times E^\Phi,\sigma\times \sigma). \]
to a map from $(D,\tau)\otimes_k L$ to $(D,\tau) \times C^+_{D,\tau},\tau)$ induces the map $\alpha$ (as opposed to say $\alpha')$.

Recall that if we are working over a global field then by Corollary \ref{cor:uniquealphaglobal} there is a unique choice of $\alpha$ for any triple $(L,D,\tau)$ which admits any $\alpha$, in particular in this context all $\beta$ induce the same $\alpha$.
\end{itemize}
The rational conjugacy classes of $T_{E,\sigma,\beta}$ in $\Spin_{(L,D,\tau,\alpha)}$ are classified by:
\begin{itemize}
\item The rational conjugacy class of the torus $T_{E,\sigma}$ in $\SO_{(D,\tau)}$.
\item The rational conjugacy class of the torus $T_{E^\Phi,\sigma}$ in $\Spin_{(D,\tau)}$.
Note that the chosen torus must have image through the natural map the the chosen torus in $\SO_{(D,\tau)}$.
\end{itemize}
\end{thm}
\begin{proof}
The theorem is an immediate consequence of Lemmas \ref{lem:structureofT} and \ref{lem:existenceofT}.
\end{proof}

\begin{rmk}
Criterion under which  $(E,\sigma)$ is associated to a torus in $\SO_{(D,\tau)}$ are discussed in Section \ref{subsec:tori-clasic}, specifically Theorem \ref{thm:invariants_upgraded} or in the case of global fields Theorem \ref{thm:the_result_upgraded}.
The rational conjugacy classes of the torus $T_{E,\sigma}$ in $\SO_{(D,\tau)}$ are characterized in Theorem \ref{thm:lambdafortoriorth}.
The rational conjugacy classes of tori $T_{E^\Phi,\sigma}$ in $\Spin_{(D,\tau)}$  covering a particular torus are characterized in Theorem \ref{thm:CONGTORISPINCOVERS}.
\end{rmk}

\subsubsection{Explicit Combinatorics of the Galois Sets for $E$ and $E^\Phi$ for $D_4$}\label{subsec:explicitcombin}

The purpose of this section is to to describe the explicit combinatorics of the Galois set structure for $E$ and $E^\Phi$ for the case of tori in groups of $D_4$.
This allows for a more concrete understaning of the condition of which $E$ can admit maps:
\[ E \otimes_k L \rightarrow E \oplus E^\Phi \]
and it will allow for a semi-explicit construction of twisted composotions.

For notational convenience we shall write the set:
\[ \Hom_L(E,\overline{L}) = \{ 1,2,3,4,-1,-2,-3,-4 \} \]
We note that the first that the Galois group $\Gal(\overline{L}/L)$ action on this set factors through $S_4 \ltimes \{\pm1\}^4$, we shall eventually make use of a more explicit presentation of this group.
Next we look at $\Hom_L(E^\Phi, \overline{L})$, we shall characterize the sets $\phi$ by specifying $\phi\cap\{1,2,3,4\}$
\begin{center}
\begin{tabular}{cccccccc}
\{1\}&\{134\}&\{124\}&\{123\}&\{234\}&\{2\}&\{3\}&\{4\}\\
\{1234\}&\{34\}&\{24\}&\{23\}&$\emptyset$ &\{12\}&\{13\}&\{14\}
\end{tabular}
\end{center}
Note that we shall make some explicit use of the ordering of this set into rows and columns.
In particular, without loss of generality we may designate the first row as $\Phi_+$ and the second as $\Phi_-$.

\begin{prop}\label{prop:dubref1}
Suppose $L=k\oplus k\oplus k$ so that $\Delta_L = k \oplus k$, and $(E,\sigma)$ is an algebra over $k$ whose discriminant is a square.
Then:
\begin{enumerate}
\item $E^\Phi$ splits over $k$ as $E^{\Phi_+} \oplus E^{\phi_-}$,
\item $(E^{\Phi_+})^\Phi \simeq E \oplus  E^{\phi_-}$, and
\item $(E^{\Phi_-})^\Phi \simeq E \oplus  E^{\phi_+}$.
\end{enumerate}
In particular if we interpret $E\oplus E^{\Phi_+} \oplus E^{\phi_-}$ as an algebra over $L$ then:
\[ (E\oplus E^{\Phi_+} \oplus E^{\phi_-})^\Phi \simeq (E\oplus E^{\Phi_+} \oplus E^{\phi_-}) \otimes_k \Delta_L \]
as a twisted map of $L$-algebras.
\end{prop}
\begin{proof}
The first point is clear.
The key observation behind (2) and (3) is that there are two types of $\sigma$-types for $E^{\Phi_+}$ and $E^{\Phi_-}$.
For for each $\rho\in \Hom_k(E,\overline{k}) $ there is a $\sigma$-type of $E^{\Phi_+}$ given by:
\[ \{ \phi \in \Hom_k(E^{\Phi_+},\overline{k}) \;\mid\; \rho\in \phi \} \]
and for each $\phi' \in \Hom_k(E^{\Phi_-},\overline{k})$ there is:
\[ \{ \phi \in \Hom_k(E^{\Phi_-},\overline{k}) \;\mid\;  \abs{\phi\cap \phi'} = 1 \}. \]
The final claim now follows immediately from the definition of the reflex algebra.
\end{proof}

\begin{prop}\label{prop:dubref2}
Suppose $L=k\oplus \Delta_L$ and $(E,\sigma)$ is an algebra over $k$ whose discriminant is that of $\Delta_L$.
Then as an algebra over $\Delta_L$ we have:
\[ (E^\Phi)^\Phi = E\otimes_k \Delta_L \oplus E^\Phi \otimes_k \Delta_L \]
In particular if we interpret $E\oplus E^{\Phi}$ as an algebra over $L$ then:
\[ (E\oplus E^{\Phi})^\Phi \simeq (E\oplus E^{\Phi}) \otimes_k \Delta_L \]
as a twisted map of $L$-algebras.
\end{prop}
\begin{proof}
The proof is as in the previous case, the only complexity is the correct interpretation of reflex algebras over algebras which are not fields.
\end{proof}

In what follows we shall implicitly focus on the case where $L$ is a field, however everything that we are saying translates over to the more general case under the correct interpretations.
The case where $L$ is not a field shall not really be needed in the sequel as the various $D_4$ groups in this case admit classical descriptions.

We shall now describe the possible Galois set structures on $E$ and $E^\Phi$, which are both once again algebras over $L$, a cubic \'etale extension of $k$.

Consider first the subgroup $H$ generated by:
\begin{align*}
a & = (12)(34) \quad &A = (1,1,-1,-1) \\
b & = (13)(24) \quad &B = (1,-1,1,-1) \\
c & = (14)(23) \quad &C = (1,-1,-1,1) \\
&&Z = (-1,-1,-1,-1)
\end{align*}
Notice these satisfy the relations:
\[ xX = Xx \qquad xY = ZYx \qquad Zx = xZ \qquad ZX = XZ \]
for $x\neq y$ with $x,y$ lower case and $X,Y$ upper case.

Denote next by $\delta_{1} = (-1,1,1,1)$, $\delta_{2} = (1,-1,1,1)$, $\delta_{3} = (1,1,-1,1)$, and $\delta_{4} = (1,1,1,-1)$.

The outer automorphism group of $H$ is $S_3\times S_3$, this group act by permuting respectively the rows and columns of:
\[ \begin{matrix} a & A & ZaA \\
                            b & B & ZbB \\
                            c & C & ZcC \\ \end{matrix}\]
The first $S_3$ factor can be interpreted as the $S_3$ subgroup $S_4$ which stabilizes $\{1\}$ and acts on $\{2,3,4\}$.

The element of the second $S_3$ factor which interchanges the first and third columns is $\delta_1$. Denote by $\tau$ the three cycle which sends the first column to the second, the second to the third and the third to the first.

\begin{lemma}
The action of the Galois group $\Gal(\overline{L}/L\otimes_k\Delta_L)$ on both $\Hom_L(E,\overline{L})$ and $\Hom_L(E^\Phi,\overline{L})$ factors through the simultaneous action of the group $G = (S_3\times \langle 1 \rangle)\ltimes H$.

If $E\otimes_k L \simeq E \oplus E^{\Phi}$ then the action of $\Gal(\overline{k}/k)$ on $\Hom_k(E,\overline{L})$ factors through some action of the group $(S_3 \times S_3)\ltimes H$.
\end{lemma}
\begin{proof}
The first point follows by observing that these are the elements of the automorphism group of $(E,\sigma) \otimes_L \overline{L}$

The second point follows by observing that these are the automorphisms of $(E,\sigma, \beta) \otimes_k \overline{k})$ for $\beta$ a map $(E^\Phi,\sigma) \rightarrow (E,\sigma)\otimes_k \Delta_L$.
\end{proof}

Notice that
If we attempt to identify the $G$-set $ \{ 1,2,3,4,-1,-2,-3,-4 \}$ with the $G$-set which is the first row according to the listed order of presentation, we find that we must twist the action by the automorphism $\tau$.
If we instead identify  $ \{ 1,2,3,4,-1,-2,-3,-4 \}$ with the second row we must twist the action by the automorphism $\tau^2$.
In particular, we can make the collection
\begin{center}
\begin{tabular}{cccccccc}
1&2&3&4&-1&-2&-3&-4\\
\{1\}&\{134\}&\{124\}&\{123\}&\{234\}&\{2\}&\{3\}&\{4\}\\
\{1234\}&\{34\}&\{24\}&\{23\}&$\emptyset$ &\{12\}&\{13\}&\{14\}
\end{tabular}
\end{center}
into a $G$-set by assigning the $S_3$ new factor to interchange rows.
We shall prove in Proposition \ref{prop:twistingactions} that up to automorphisms of $E^\Phi$ this is the identification created by $\beta$.

In order to characterize the options we shall want to understand the possible images of $\Gal(\overline{k}/k)$ in $(S_3\times S_3)\ltimes H)$ and as such characterize the subgroups. The following Lemmas shall be helpful.

\begin{lemma}
If $L$ is a field, and there is a map $\beta: E^\Phi \rightarrow E\otimes_k \Delta_L$ then $E \simeq E^\sigma \oplus E^\sigma$ with $\sigma$ acting as the exchange involution if and only if the image of $\Gal(\overline{k}/k)$ is in $(S_3\times S_3)$. That is, the image contains no elements of $H$.
\end{lemma}
This is immediate.

\begin{lemma}
If $L$ is a field, and there is a map $\beta: E^\Phi \rightarrow E\otimes_k \Delta_L$then up to relabelling the elements of $ \{ 1,2,3,4,-1,-2,-3,-4 \}$ the image of $\Gal(\overline{k}/k)$ in $(S_3\times S_3)\ltimes H$ contains one of $(1\times \tau )$ or $((2 3 4) \times \tau )$ in $S_3 \times S_3$.

If the image in $S_3 \times S_3$ contains $((2 3 4) \times \tau )$ and any other non-trivial element than the image contains $(1\times \tau )$.
\end{lemma}
\begin{proof}
First notice that the inner automorphisms of $(S_3\times S_3)\ltimes H$ induced by $(S_3\times \langle \delta_1\rangle)\ltimes H$ all correspond to relabellings of $ \{ 1,2,3,4,-1,-2,-3,-4 \}$.
Second, notice that if the image contains $Z$ times one of the above then it contains the element.

Next, by observing 
\[ x\tau x = Xx \tau  \quad X\tau X = Zx \tau \quad (ZxX) \tau (ZxX) = ZX\tau x \quad xY\tau Yx = xZyX \tau  \quad yYX\tau XYy = Yx \tau   \]
we see that we can relabel so as to ensure that the element that acts as a three cycle on $\Hom_k(L,\overline{k})$ has image in $(S_3\times S_3)\times \langle Z \rangle $.

Relabellings using the $S_3$ factor complete the result.
\end{proof}

\begin{lemma}
Suppose $L$ is a field and there is a map $\beta: E^\Phi \rightarrow E\otimes_k \Delta_L$.
Suppose further that the image of $\Gal(\overline{k}/k)$ contains $\tau$, then up to relabellings of $\{2,3,4\}$ the image of $\Gal(\overline{k}/k)$ in $H$ is one of:
\[ \{1\},\quad \{ 1,Z \},\quad \{1, Za, ZA, aA \},\quad \{ 1, Z, a, Za, A, ZA, aA, ZaA \},\quad H. \]
If $L$ is not Galois then the option $\{1, Za, ZA, aA \}$ does not occur.
If the image in $S_3\times S_3$ is larger than $\langle (3 4)\rangle \times S_3$ then options $\{1, Za, ZA, aA \}$ and $ \{ 1, Z, a, Za, A, ZA, aA, ZaA \}$ do not occur. 

If the image of $\Gal(\overline{k}/k)$ contains $\tau \times (2 3 4)$, then up to relabellings of $\{2,3,4\}$ the image of $\Gal(\overline{k}/k)$ in $H$ is one of:
\[ \{ 1\}, \quad \{ 1, Z \}, \quad \langle  Z, a,  B,  cC \rangle, \quad H. \]
If the image in $S_3\times S_3$ is larger than $\langle   (2 3 4) \times \tau\rangle$ then option $\langle  Z, a,  B,  cC \rangle$ does not occur.
\end{lemma} 
\begin{proof}
First, note that if we have any two of $a,b,c,Za,Zb,Zc$ which are not $Z$ multiples then we will generate the whole group.

By using the orbit structure of $\tau$ and $\tau \times (2 3 4)$ in $H$ a simple case analysis covers the other possibilities.
\end{proof}

\begin{rmk}
The above does not exhaust all the limitations we can place on the Galois group. For example in a most of the cases we can show that the image $J \subset S_3\times S_3$ has a section of the form $J\ltimes \{1\}$ inside the total image of the group in $(S_3\times S_3)\ltimes H$.
\end{rmk}

\begin{claim}
If the image of $\Gal(\overline{k}/k)$ contains $\tau$ and the image in $H$ is:
\begin{itemize}
\item $\{ 1\}$ then $E = L \oplus L \oplus F \oplus F$ where $F$ is cubic over $k$, and $\sigma$ acts as to exchange factors.
\item $\{ 1, Z \}$  then $E = (L\oplus F) \otimes K$ where $F$ is cubic over $k$ and $K$ is quadratic over $k$, and $\sigma$ acts on $K$.
\item $\{1, Za, ZA, aA \}$ then $E = K_1 \otimes K_2 \oplus K_2 \otimes K_3$ where on one side $\sigma$ acts through $K_1$ and the other side through $K_2$. The action of the (cyclic) Galois group of $L$ is to permute the three quadratic subextensions of $K_1 \otimes K_2$. If the element$((34) \times 1)\ltimes 1$ is in the image it acts on $K_3$.
\item $ \{ 1, Z, a, Za, A, ZA, aA, ZaA \}$ then $E = K_1 \otimes K_2 \oplus K_3 \otimes K_4$ where sigma acts on $K_1$ and $K_3$. The action of the Galois group of $L$ is to permute the fields $K_1, K_2$ and $K_3$.
\item $H$ then $E$ is a field. If the element $((34) \times 1)\ltimes 1$ is in the image it acts on $K_4$.
\end{itemize}
If the image of $\Gal(\overline{k}/k)$ contains $ (2 3 4) \times \tau$ and not $\tau$ and the image in $H$ is:
\begin{itemize}
\item $\{ 1\}$ then $E = L \oplus L \oplus (L \otimes_k L) \oplus  (L \otimes_k L)$ and $\sigma$ acts to exchange factors.
\item $ \{ 1, Z \}$ then $E =  (L  \oplus (L \otimes_k L) ) \otimes K$ where $K$ is quadratic over $k$ and $\sigma$ acts on $K$.

\item $\langle  Z, a,  B,  cC \rangle$ then $E$ is a field, its Galois closure $\overline{E}$ over $L$ is Galois over $k$.
$E = \overline{E}^B$ is the fixed field of $B$,
$E^{\Phi_+} \simeq \overline{E}^{ZcC}$ and $ E^{\Phi_+} \simeq \overline{E}^{a}$.
There are elements $D_1,D_2,D_3 \in L$ permuted by the (cyclic) Galois group of $L$ so that
\[ L(\sqrt{D_1}) =  \overline{E}^{B,cC,Z} \qquad  L(\sqrt{D_2}) =  \overline{E}^{a,B,Z} \qquad   L(\sqrt{D_3}) =  \overline{E}^{a,cC,Z} \]
Moreover $E^\sigma =L(\sqrt{D_1})\otimes L(\sqrt{D_2})$, $(E^{\Phi_+})^\sigma =L(\sqrt{D_3})\otimes L(\sqrt{D_1})$, and $(E^{\Phi_-})^\sigma =L(\sqrt{D_2})\otimes L(\sqrt{D_3})$, in particular $\overline{E}^\sigma = \overline{E}^Z = L(\sqrt{D_1},\sqrt{D_2},\sqrt{D_3})$.
\item $H$  then $E$ is a field.
\end{itemize}
\end{claim}
The argument is tedious, and we omit the details.

\begin{prop}\label{prop:twistingactions}
With notation as above, if $\beta: E^\Phi \rightarrow E\otimes_k \Delta_L$ is an $L$-skew isomorphism then, up to  automorphisms of $E^\Phi$ the action of $\Gal(\overline{k}/k)$ on the $\Gal(\overline{k}/k)$-set $\Hom_k(E,\overline{k})$ factors through an action of $(S_3\times S_3)\ltimes H$ under an identification of the elements of $\Hom_k(E,\overline{k})$ as:
\begin{center}
\begin{tabular}{cccccccc}
1&2&3&4&-1&-2&-3&-4\\
\{1\}&\{134\}&\{124\}&\{123\}&\{234\}&\{2\}&\{3\}&\{4\}\\
\{1234\}&\{34\}&\{24\}&\{23\}&$\emptyset$ &\{12\}&\{13\}&\{14\}
\end{tabular}
\end{center}
Moreover, the kernel of the quotient map onto the second $S_3$ factor gives $\Gal(\overline{k}/L\otimes_k\Delta_L)$, the subgroup fixing $L$ and $\Delta_L$. The element $\tau$ whose image generates $\Aut_{\Delta_L}(L\otimes_k\Delta_L)$ acts by permuting the rows.
\end{prop}
\begin{proof}
Without loss of generality we may suppose that we have labelled things to be in one of the cases of the previous lemma.

By observing that:
\[ E \otimes_k \overline{L} \simeq E\otimes_L \overline{L} \oplus (E\otimes_k \Delta_L)\otimes_{\rho,L} \overline{L}  \simeq  E\otimes_L \overline{L} \oplus E^\Phi \otimes_{\rho,L} \overline{L} \]
we see that there must exist some identification, and any such identification can be presented in terms of columns which indicate the orbits of $\tau$.

Next we observe that if we have an admissible identification, the interchanging $\phi \leftrightarrow \overline{\phi}$ for all $\phi\in \Phi$ would also give an admissible identification.
Thus, without loss of generality we may assume that $1 \in \phi_1 = \tau\cdot 1$.
We also observe that if the Galois set $\Phi$ decomposes into multiple orbits with respect to the image inside $(\langle \delta_1 \rangle \times S_3)\ltimes H$ we can perform the interchange $\phi \leftrightarrow \overline{\phi}$ separately on each orbit and still obtain an admissible identification.

We consider first the case where the image of $\Gal(\overline{k}/k)$ contains $\tau$ and $H$

Because $aA\tau = \tau ZA$ and $A\cdot 1 = 1$ we must have $aA\cdot (\tau \cdot 1) = \tau \cdot -1$.
But the only elements $\phi\in \Phi$ with $\overline{\phi} = aA\phi$ are $\{1\}$, $\{2\}$, $\{234\}$, $\{134\}$.
A similar argument with $bB$ rules out $\{2\}$ and $\{134\}$.
Thus, without loss of generality, $\tau \cdot 1$ is $\{1\}$.
Next, because $ZA\tau = \tau Za$ we can conclude that $\tau \cdot 2 = A \tau \cdot 1$ is $\{134\}$.
Similar arguments complete the second row.

A similar argument with $aZ\tau^2 = \tau^2 ZA$ allows us to conclude that $\tau^2\cdot 1$ is one of 
$\{1234\},\{12\},\emptyset,\{34\}$.
Here  $bB$ rules out $\{12\}$ and $\{34\}$.

There are now two cases, either the image of the Galois group contains $\delta_1$ (or any element which interchanges the second and third row) or it does not.
If it does not, then by the separation of $\Phi^+$ and $\Phi^-$ as orbits we can arbitrarily conjugate $\Phi^-$ and ensure that $\tau^2\cdot 1 = \{1234\}$.
However, if there was an interchange element, then because $\tau^2 = \delta_1 \tau \delta_1$ we can conclude that the third row is as given.

A tedious case analysis covers the cases where the image does not cover $H$.
For example when we do not contain $bB$, but do contain $ZA$ then there will be an automorphism of the orbit of $1$ which interchanges $\{ 1 \}$ and $\{134\}$. This will complete the diagram for $1,2,-1,-2$. A symmetric argument covers the diagram under $3,4,-3,-4$.
If the image contains only $\{1 , Z\}$ then there are typically many admissible identifications.

When instead the image contains $ (2 3 4) \times \tau$ similar tricks involving the observation that $B\cdot 1 = 1$ allow us to conclude $\tau \cdot 1 = (2 3 4) \tau \cdot B 1 = ZcC(2 3 4) \tau \cdot 1 $ for which the only options are $\{1\}$ and $\{234\}$. We are free to prefer $\{1\}$. Because the subgroup of $H$ acts transitively on $\{1,2,3,4,-1,-2,-3,-4\}$ we can immediately conclude the result for the second row.
The same argument works for the third row.
\end{proof}

\begin{rmk}
In the above if $E$ and $E^\Phi$ are fields over $k$ then the only choices being made in the normalization of the map $\beta$ are labellings and possibly pre/post composition with the automorphism $\sigma$.
More generally if $E^{\Phi^+}$ and $E^{\Phi^-}$ are fields the only choices are pre/post compositions with $\sigma$ acting separately on the two factors.
In general, all of the choices constitute pre/post composition with automorphisms of $E^{\Phi}$ as an algebra with involution.

The presentation above seems to makes many non-canonical choices, in particular the ordering and labelling of the Galois sets.
A labelling compatible with the one above can be selected after one considers an appropriate map $\beta$.
For example, if $\Delta_L = k\oplus k$ and $E=L^8$, so that $E^\Phi = L^{16}$, there may exist many maps $\beta$.
\end{rmk}

\begin{prop}\label{prop:dubref3}
With notation as above, if $\beta: E^\Phi \rightarrow E\otimes_k \Delta_L$ is an $L$-skew isomorphism then the set $\Hom_k(E\otimes_k \Delta_L,\overline{k})$ is identified with the Galois set:
\[  \Hom_L(E,\overline{L}) \times \{+,-\}   \sqcup   (\Phi^+ \sqcup \Phi^-) \times \{+,-\} \]
Where the Galois group acts on $\{+,-\} $ through an identification with $ \Hom_k(\Delta_L, \overline{k})$.

The set $\Hom_k(E^\Phi,\overline{k})$  is identified with the Galois set:
\[ \Phi^+ \times \{ \mu_1 \} \sqcup \Phi^-  \times \{ \mu_1 \} \sqcup \Phi^- \times \{ \mu_2 \} \sqcup \Hom_L(E^\Phi,\overline{L})  \times \{ \mu_2 \} \sqcup \Hom_L(E^\Phi,\overline{L}) \times \{ \mu_3 \} \sqcup \Phi^+ \times  \{ \mu_3 \} \]
Where the action of the Galois group on $\{ \mu_1,\mu_2,\mu_3 \}$ through an identification with $\Hom_k(L,\overline{k})$.

The map $\beta$ identifies the Galois sets under the association:
\begin{align*}  
&\Phi^+ \times \mu_1 \leftrightarrow \Phi^+ \times + 
       \qquad &\Phi^- \times \mu_1 \leftrightarrow \Phi^- \times - \\
 &\Phi^- \times \mu_2 \leftrightarrow \Phi^- \times + 
       \qquad &\Hom_L(E^\Phi,\overline{L})  \times \mu_2  \leftrightarrow \Hom_L(E^\Phi,\overline{L})  \times - \\
&\Hom_L(E^\Phi,\overline{L}) \times \mu_3 \leftrightarrow   \Hom_L(E^\Phi,\overline{L})  \times + 
       \qquad   &\Phi^+ \times   \mu_3   \leftrightarrow \Phi^+ \times -.
\end{align*}
\end{prop}
This follows as in Proposition \ref{prop:dubref1} and \ref{prop:dubref2} by base change.

\begin{lemma}\label{lem:taurelation}
With notation as above $\tau$ translates  relations between $  \Hom_L(E^\Phi,\overline{L})$, $\Phi^+$ and $\Phi^-$ as follows.
Let $\rho \in   \Hom_L(E^\Phi,\overline{L}),\; \phi_+ \in \Phi^+$ and $\phi_- \in \Phi^-$ then:
\begin{itemize}
\item $\rho \in \phi^+$ if and only $\abs{\tau(\rho)  \cap \tau(\phi_+)} = 1$
\item $\rho \not\in \phi^+$ if and only $\abs{\tau(\rho)  \cap \tau(\phi_+)} = 3$
\item $\rho \in \phi^-$ if and only if $\tau(\phi^-) \in \tau(\rho)$
\item $\rho \not\in \phi^-$ if and only if $\tau(\phi^-) \not\in \tau(\rho)$
\item $\abs{\phi_+\cap \phi_-} = 1$ if and only if $\tau(\phi_-) \in \tau(\phi_+)$.
\item $\abs{\phi_+\cap \phi_-} = 3$ if and only if $\tau(\phi_-) \not\in \tau(\phi_+)$.
\end{itemize}
From which it follows further that $\{\rho\} = \phi_+ \cap \phi_-$ if and only if $\{ \tau(\phi_-) \} = \tau(\rho) \cap \tau(\phi_+)$.
\end{lemma}
This is a direct check.

We shall now work to construct a map from $E$ to $E$ which shall act like a twisted composition.

First recall the map $\Psi :E^\Phi \rightarrow \tilde{E} \subset E\otimes_L E^\Phi$. We may write the definition as:
\[ \Psi\left( \underset{\phi}\sum a_\phi e_\phi \right) = \underset{\phi_1\cap\phi_2  = \{\rho\}}\sum a_{\phi_1}a_{\phi_2} e_\rho\otimes (e_{\phi_1}+e_{\phi_2} + e_{\overline{\phi}_1}+e_{\overline{\phi}_2}). \]
In our present case, we note that for each $\rho$, there are precisely $4$ pairs $\phi_1$,$\phi_2$ with $\phi_1\cap \phi_2 = \{\rho\}$.
It follows that the image of $\Psi$ is contained in the $32$ dimensional subalgebra $\tilde{E}$ of $E\otimes_L E^\Phi$ whose idempotent are $e_\rho \otimes (e_{\phi_1} + e_{\phi_2}+e_{\overline\phi_1} + e_{\overline\phi_2})$ for $\phi_1\cap\phi_2 = \{\rho\}$. This algebra $\tilde{E}$ is a degree $4$ algebra over $E$.

\begin{rmk}
In light of Lemma \ref{lem:taurelation} and Proposition \ref{prop:dubref3} the definition which we have just given as a map over $L$ can also be expressed as a map over $L$ as:
\[
 \Phi \left( \sum_{\rho,\mu} a_{\rho,\mu} e_{\rho,\mu} + \sum_{\phi,\mu} a_{\phi,\mu} e_{\phi,\mu} \right) =  
    \underset{\phi_1\cap\phi_2  = \{\rho\}} \sum \left( 
\begin{matrix}
a_{\phi_1,\mu_1}a_{\phi_2,\mu_1} e_{\rho,\mu_1}\otimes (e_{\phi_1,\mu_1}+e_{\phi_2,\mu_1}+e_{\overline{\phi}_1,\mu_1}+e_{\overline\phi_2,\mu_1}) + \\
a_{\phi_2,\mu_2}a_{\rho,\mu_2} e_{\phi_1,\mu_2}\otimes (e_{\rho,\mu_2}+e_{\phi_2,\mu_2}+e_{\overline\rho,\mu_2}+e_{\overline\phi_2,\mu_2}) + \\
a_{\rho,\mu_3}a_{\phi_1,\mu_3} e_{\phi_2,\mu_3}\otimes (e_{\rho,\mu_3}+e_{\phi_1,\mu_3}+e_{\overline\rho,\mu_3}+e_{\overline\phi_1,\mu_3}) 
\end{matrix} \right). 
\]
With this definition, we can observe that the torus $T_{E,\sigma,\beta}$ is precisely:
\[ T_{E,\sigma,\alpha}(R) = \{ x \in (E\otimes_k R)^\times \;\mid\; x\sigma(x) = 1,\; x \otimes 1 = \Psi \circ \alpha^{-1}(x) \}. \]
\end{rmk}

Next, the algebra $\tilde{E}$ admits a natural $L$ linear to $E$ by.
\[ 
Tr_{\tilde{E}/E}\left(\sum a_{\phi_1,\phi_2} e_\rho\otimes (e_{\phi_1}+e_{\phi_2}+e_{\overline{\phi}_1}+e_{\overline{\phi}_2})\right)  = \sum a_{\phi_1,\phi_2} e_\rho
\]
this is precisely the trace map from $\tilde{E}$ to $E$.

Now, given an element $\Lambda$ in $\tilde{E}$ we obtain a map $\twistedcomp_\Lambda : E \rightarrow E$ 
through the composition of maps:
\[ E \longrightarrow E\otimes_k \Delta_L \overset{\beta^{-1}}\longrightarrow E^\Phi \overset{\Psi}\longrightarrow \tilde{E} \overset{[\Lambda]}\longrightarrow \overset{Tr_{\tilde{E}/E}}\longrightarrow E. \]

This map is explicitly given as:
\[ \sum_{\phi_+\cap \phi_- = \{\rho\}}   \Lambda_{\rho,\phi_+,\phi_-,\mu_1} a_{\phi_+}a_{\phi_-} e_{\rho} +  \Lambda_{\phi_-,\rho,\phi_+,\mu_2} a_{\phi_-}a_{\rho} e_{\phi_+}+  \Lambda_{\phi_+,\phi_-,\rho,\mu_3} a_{\rho}a_{\phi_+} e_{\phi_-}. \]

\begin{lemma}
Using the notation above, 
for all $x\in E$ and for all
 $\ell \in L$ we have:
\[ \ell \twistedcomp_1( \ell x) = N_{L/k}(\ell) \twistedcomp_\Lambda(x). \]
For all $e \in T_{E,\sigma,\beta}(R)$ we have:
\[ e\twistedcomp_\Lambda(x) = \twistedcomp_\Lambda(e x). \]
Moreover, every quadratic map from $E$ to $E$ satisfying these properties is of the form $\twistedcomp_\Lambda$ for some choice of $\Lambda$.
\end{lemma}
\begin{proof}
The first claim is a direct check by observing that $\ell$ acts on $a_{\rho},e_{\rho}$ through $\mu_1(\ell)$ on $a_{\phi_+},e_{\phi_+}$ through $\mu_2(\ell)$ and $a_{\phi_-},e_{\phi_-}$ through $\mu_3(\ell)$.

The second claim follows immediately from the definition of $ T_{E,\sigma,\beta}$ and the effective linearity of $\twistedcomp_\Lambda$.

The third claim follows immediately by a comparison of eigenspaces.
\end{proof}

\begin{cons}\label{cons:twist1}
With notation as above, consider the algebra:
\[ V = E \oplus (L^8 \otimes \Delta_L) \]
as a module over:
\[ M = L \oplus L\otimes \Delta_L. \]
Let $\tilde{Q}$ be a $\Delta_L$ quadratic form on $(L^8 \otimes \Delta_L)$ and let $\lambda \in E^\sigma$, and be such that the $M$ quadratic form:
\[ \hat{Q} = \Tr_{E^\sigma/L}(\hat{\lambda} x\sigma(x)) \oplus \tilde{Q} \]
admits the structure of a twisted composition over $M$. That is, suppose that the Clifford invariant of the form is trivial.
Let $(M,V,\hat{Q},\hat\twistedcomp)$ be such a twisted composition.

Then $\hat\twistedcomp$ induces a quadratic map $\tilde{\Phi}:E^\Phi \rightarrow E$:
\[  \sum_{\phi_+\cap \phi_- = \{\rho\}}   \Lambda_{\rho,\phi_+,\phi_-} a_{\phi_+}a_{\phi_-} e_{\rho} \]
for some $\Lambda \in \tilde{E}$ as follows:

Firstly, because we are dealing with twisted composition which does not come from a field we have that $\SO_{E,Q|_E}$ is the image of $\Spin_{(M,F,\hat{Q},\hat\twistedcomp)}$.
Thus the inclusion of $T_{E,\sigma,\beta} \rightarrow \Spin_{(M,F,\hat{Q},\hat\twistedcomp)}$ gives the spin representation space (that is $(L^8 \otimes \Delta_L)$) the structure of a free $E^\Phi$ module, compatible with the $\Delta_L$ structure. 
Identify $(L^8 \otimes \Delta_L)$ by picking a basis element.
It follows that $\hat\twistedcomp$ is actually inducing a map quadratic map from $E^\Phi$ to $E$.
The fact that $T_{E,\sigma,\beta}$ is a torus in $\Spin_{(M,F,\hat{Q},\hat\twistedcomp)}$ implies the form for the map is as in the previous lemma.
Moreover, it follows that $\tilde{Q}$ is of the form:
\[ \Tr_{E^\Phi/\Delta_L}( \tilde{\lambda} x\sigma(x)) \]
for some $\lambda\in (E^\Phi)^\sigma$.

Next, because $\twistedcomp$ is actually a twisted composition, it follows that if we write:
\[ Q =  \Tr_{E^\sigma/L}(\hat{\lambda} x\sigma(x)) \]
the map $\tilde{\Phi}$ satisfies:
\[ Q \circ \tilde{\Phi}(x) = N_{\Delta_L/L}( \tilde{Q}(x)). \]

To see the consequences of this we expand out both sides. First:
\[ Q \circ \tilde{\Phi}(x) 
= \sum_{\{\rho,\overline{\rho}\}} 
\left(\sum_{\phi_1\cap\phi_2 = \{\rho\}} \Lambda_{\rho,\phi_1,\phi_2} a_{\phi_1}a_{\phi_2} \right)
\left(\sum_{\phi_4\cap\phi_3 = \{\overline{\rho}\}} \Lambda_{\overline{\rho},\phi_3,\phi_4} a_{\phi_3}a_{\phi_4} \right) \]
Expanding out the product his sum decomposes into two parts, those where $\{\overline{\phi_1},\overline{\phi_2}\} = \{\phi_3,\phi_4\}$ and those where it does not. There are exactly $16$ terms where $\{\overline{\phi_1},\overline{\phi_2}\} = \{\phi_3,\phi_4\}$ and $48$ where they are not equal. These latter come naturally in pairs as there are exactly $24$ configurations of $\{ \phi_1,\phi_2,\phi_3,\phi_4\}$ which can arise this way.

In particular, if  $\{\overline{\phi_1},\overline{\phi_2}\} \neq \{\phi_3,\phi_4\}$ then a term containing:
\[ a_{\phi_1}a_{\phi_2}a_{\phi_3}a_{\phi_4} \]
appears for exactly $2$ different $\{\rho,\overline{\rho}\}$. 

If we expand out the right hand side we obtain only terms of the first type.
Thus the coefficients  $\Lambda_{\rho,\phi_1,\phi_2}$ ensure that these duplicate terms of the second type cancel in the sum.

Moreover, what remains on the left hand side after this cancellation is the sum:
\[  \sum_{\{ \phi_1,\phi_2,\phi_3,\phi_4\}} \lambda_\rho\Lambda_{\rho,\phi_1,\phi_2}\Lambda_{\rho,\phi_3,\phi_3} a_{\phi_1}a_{\phi_2}a_{\phi_3}a_{\phi_4}  \]
Over the $16$ terms $\{ \phi_1,\phi_2,\phi_3,\phi_4\}$ with $\phi_1\cap \phi_2 = \overline{\phi_3\cap\phi_4} = \{\rho\}$, $\phi_1=\overline{\phi_3}$ and $\phi_2=\overline{\phi_4}$.
On the right hand side the sum is precisely:
\[  \sum_{\{ \phi_1,\phi_2,\phi_3,\phi_4\}} \lambda_{\phi_1}\lambda_{\phi_2} a_{\phi_1}a_{\phi_2}a_{\phi_3}a_{\phi_4}.  \]
We thus conclude that:
\[  \lambda_\rho\Lambda_{\rho,\phi_1,\phi_2}\Lambda_{\rho,\phi_3,\phi_3} =  \lambda_{\phi_1}\lambda_{\phi_2}. \]
\end{cons}

\begin{lemma}
If $L$ is a field, then to ensure that:
\[ (L,E,\Tr(\lambda x\sigma(x)), \twistedcomp_\Lambda ) \]
is a twisted composition it suffices to ensure that  when we expand
\[\sum_{\{\rho,\overline{\rho}\}} \left( \sum_{\phi_+\cap \phi_- = \{\rho\}}   \Lambda_{\rho,\phi_+,\phi_-,\mu_1} a_{\phi_+}a_{\phi_-} \right)\left( \sum_{\phi_+\cap \phi_- = \{\overline{\rho}\}}   \Lambda_{\overline\rho,\phi_+,\phi_-,\mu_1} a_{\phi_+}a_{\phi_-} \right) \]
as:
\[ \sum_{(\phi_1,\phi_2,\phi_3,\phi_4)} \lambda_\rho \Lambda_{\rho,\phi_1,\phi_2,\mu_1}\Lambda_{\overline\rho,\phi_3,\phi_4,\mu_1}a_{\phi_1}a_{\phi_2} a_{\phi_3}a_{\phi_4} \]
consisting of terms 
that:
\begin{enumerate}
\item The terms where $\phi_1\cap\phi_2 = \overline{\phi_3\cap\phi_4} = \{\rho \}$ each term where $\abs{\phi_1\cap\phi_3}=2$ cancels with the other term contributing $a_{\phi_1}a_{\phi_2} a_{\phi_3}a_{\phi_4}$ ( specifically the term $(\phi_1,\phi_4,\phi_2,\phi_3)$)
\item The  terms  for which $\phi_1=\overline{\phi_3}$ and $\phi_2=\overline{\phi_4}$ satisfy 
\[  \lambda_\rho \Lambda_{\rho,\phi_1,\phi_2,\mu_1}\Lambda_{\overline\rho,\phi_3,\phi_4,\mu_1} = \lambda_{\phi_1}\lambda_{\phi_2} \]
\end{enumerate}
\end{lemma}
\begin{proof}
The only content is that it suffices to check on the image over $\mu_1$. This follows from the fact that because $L$ is a field  the action of the Galois group which permutes embeddings of $L$ also permutes the formulas which need to be checked.
The conditions being checked are Galois stable, and thus checking for a single $\mu_i$ suffices.
\end{proof}

\begin{cons}\label{cons:twist2}
With notation from the previous example,
if the elements $\lambda = (\hat{\lambda}, \tilde{\lambda}) \in E^\sigma \otimes_k L$ is actually in $E^\sigma \otimes 1$ then with $\Lambda$ as above:
\[ (L,E,\Tr(\lambda x\sigma(x)), \twistedcomp_\Lambda ) \]
is a twisted composition.

This is immediate because $\Lambda_{\rho,\phi_1,\phi_2} = \Lambda_{\rho,\phi_1,\phi_2,\mu_1}$.
\end{cons}

\begin{lemma}
If we have a datum $ (\hat{\lambda}, \tilde{\lambda}, \Lambda) $  which gives a twisted composition on $E \oplus E^\Phi$ over $L$ then for any $e \in (E^\Phi)^\times$ so too does the datum  $ (\hat{\lambda}, e\tilde{\lambda}, \Psi(e)\Lambda) $ 
\end{lemma}
\begin{proof}
It is an immediate check that a change of the sort either:
\[ (\hat{\lambda}, e\tilde{\lambda}, \Lambda) \qquad \text{or} \qquad (\hat{\lambda}, \tilde{\lambda}, \Psi(e)\Lambda) \]
does not effect the terms which cancel.
It is also immediate that making both changes at the same time precisely preserves the appropriate equality between terms which do not cancel.
\end{proof}

\begin{cons}\label{cons:twist3}
If $ (\hat{\lambda}, \tilde{\lambda}, \Lambda) $ is any datum as constructed in Construction \ref{cons:twist1} then the datum
\[
  \left(\hat{\lambda}, \alpha^{-1}(\hat{\lambda}), \Phi\left(\frac{ \alpha^{-1}(\hat{\lambda})}{ \tilde{\lambda}}\right)\Lambda\right) 
\]
satisfies the conditions of Construction \ref{cons:twist2} and thus with 
\[ \hat{\Lambda} =  \Phi\left(\frac{ \alpha^{-1}(\hat{\lambda})}{ \tilde{\lambda}}\right)\Lambda \]
 we have that 
\[ (L,E,\Tr(\hat{\lambda} x\sigma(x)), \twistedcomp_{\hat\Lambda} ) \]
is a twisted composition.
\end{cons}

\begin{lemma}
If we have a datum $ (\hat{\lambda}, \tilde{\lambda}, \Lambda) $  which gives a twisted composition on $E \oplus E^\Phi$ over $L$ then for any $e \in E$ for which $\Psi(\alpha^{-1}(e)) = e\otimes 1$  so too does the datum 
\[  \left(e\sigma(e)\hat{\lambda}, e\sigma(e)\tilde{\lambda}, \Phi\left(\frac{ \alpha^{-1}(\hat{\lambda})}{ \tilde{\lambda}}\right)(\frac{1}{e}\otimes 1)\Lambda \right). \]
\end{lemma}
This is an immediate check.

\begin{rmk}
In general it is hard to modify a datum by making an arbitrary modification to the first coordinate.
This reflects the fact that for some choices, there would be no possible correction.
\end{rmk}

\begin{lemma}
Consider the algebra $\tilde{E}^\sigma$ and the map:
\[ \pi_1 : E^\sigma \times \tilde{E} \rightarrow \tilde{E}^\sigma \]
given by $(e_1, e_2) \mapsto  e_1e_2\sigma(e_2)$.
If we have a datum $ (\hat{\lambda}, \tilde{\lambda}, \Lambda) $  which gives a twisted composition on $E \oplus E^\Phi$ over $L$ then for $(e_1,e_2)\in  E^\sigma \times \tilde{E}$ the datum
\[  \left( e_1\hat{\lambda}, \alpha^{-1}(e_1)\tilde{\lambda},e_2\Lambda \right). \]
Satisfies compatibility conditions on the terms which don't cancel if and only if $\pi_1(e_1, e_2) = 1$.
\end{lemma}
This is an immediate check.

We have a decomposition of the algebra:
\[ E \otimes_L E \simeq E \oplus F_1 \]
where $F_1$ is rank $24$ \'etale algebra whose idempotents over $\overline{L}$ are naturally indexed by the ordered pairs $(x,y) \in \{1,2,3,4,-1,-2,-3,-4\}^2$ for which $x\neq y$.
This algebra admits a canonical involution induced by interchanging factors (equivalently $(x,y)\leftrightarrow (y,x)$).

\begin{lemma}
Consider the algebra $F_1$ and the map:
\[ \pi_2 : E^\sigma \times \tilde{E} \rightarrow F_1 \]
given by :
\[ \pi_2\left( \sum_{\phi_1\cap\phi_2 = \{\rho\}} \Lambda_{\rho,\phi_1,\phi_2} a_{\phi_1}a_{\phi_2} \right) = 
    \sum_{(x,y)} \frac{ \lambda_{\rho_1}\Lambda_{\phi_1,\phi_2}\Lambda_{\phi_3,\phi_4}}{ \lambda_{\rho_2}\Lambda_{\phi_1,\phi_4}\Lambda_{\phi_3,\phi_2}} e_{(x,y)} \]
where $\phi_1\cap\phi_3 = \{\rho_x,\rho_y\} = \overline{\phi_2\cap\phi_4}$, $\phi_1\cap \phi_2 = \{\rho_1\}$, and $\phi_1\cap \phi_4 = \{\rho_2\}$.
If we have a datum $ (\hat{\lambda}, \tilde{\lambda}, \Lambda) $  which gives a twisted composition on $E \oplus E^\Phi$ over $L$ then for $(e_1,e_2)\in  E^\sigma \times \tilde{E}$ the datum
\[  \left( e_1\hat{\lambda}, \alpha^{-1}(e_1)\tilde{\lambda},e_2\Lambda \right). \]
Satisfies cancellations conditions to be a twisted composition if and only if  $\pi_2(e_1, e_2) = 1$.
\end{lemma}

There is a decomposition of the algebra:
\[  E^\sigma \otimes_L E^\sigma \simeq E^\sigma \oplus F_2 \]
where $F_2$ is rank $12$ \'etale algebra whose idempotents over $\overline{L}$ are indexed by ordered pairs $(x,y) \in \{1,2,3,4\}^2$ for which $x\neq y$.
This algebra admits a canonical involution induced by interchanging factors (equivalently $(x,y)\leftrightarrow (y,x)$.

Both $\tilde{E}^\sigma$ and $F_1$ admit maps to $F_2$, the maps $\pi_1$ and $\pi_2$ induce a map:
\[ \pi_3 : T_{\tilde{E}} \overset{\pi_1\times \pi_2}\longrightarrow T_{\tilde{E}^\sigma} \times_{T_{F_2}} T_{F_1}. \]

\begin{lemma}
If we have a datum $ (\hat{\lambda}, \tilde{\lambda}, \Lambda) $  which gives a twisted composition on $E \oplus E^\Phi$ over $L$ then for $(e_1,e_2)\in  E^\sigma \times \tilde{E}$ the datum
\[  \left( e_1\hat{\lambda}, \alpha^{-1}(e_1)\tilde{\lambda},e_2\Lambda \right). \]
gives a twisted composition if and only if $\pi_3(e_1, e_2) = 1$.
\end{lemma}

Define $M$ to be the kernel of the map $\pi_3$.

\begin{cor}
Twisted compositions containing for which the spin group contains $T_{(E,\sigma)}$ are in bijection with $M(k)$.
\end{cor}

\begin{rmk}
It would be interesting to have a more complete description of $M(k)$ as well as information about the equivalences relation on it.
In particular it would also be interesting to understand how elements of $M(k)$ change the isomorphism classes of the associated twisted compositions. We shall not give a complete description here.

For local and global fields, the very restrictive classification of twisted compositions makes this a somewhat less interesting question, in particular as $H^1(\Spin)$ typically vanish, isomorphism classes are more easily studied in this case (see Example \ref{ex:twistedsig}).

More generally Proposition \ref{prop:CENTERONTWISTED} describes how the cohomology of the center acts to effect isomorphism classes. Aside from the action of the center, the isomorphism classes of twisted compositions should primarily be determined by the datum $(L,V,Q)$.
\end{rmk}

\subsubsection{Tori for Twisted Compositions}

The following is an immediate corollary of Theorem \ref{thm:bigone} together with the construction of trialitarian algebras from twisted compositions (see \cite[36.19]{book_of_involutions}). 
\begin{cor}\label{cor:twistedtoricase}
Let $(L,V,Q,\twistedcomp)$ be a twisted composition and let $T \subset \Spin_{(L,V,Q,\twistedcomp)}$ be a maximal torus.
Then 
\[ T \injects \Spin_{(L,V,Q,\twistedcomp)} \injects \Res_{L/k}(\SO_Q) \]
and the centralizer of the image of $T$ is a maximal torus in $ \Res_{L/k}(\SO_Q)$ and thus
\[  T \injects  \Res_{L/k}(T_{E,\sigma}) \] for $(E,\sigma)$ an \'etale algebra with involution over $L$.
\begin{itemize}
\item The algebra $E$ has the same discriminant over $L$ as $L$ does over $k$.

\item There exists an isomorphism $\beta : (E\otimes_L \Delta_L, \sigma\otimes_L 1) \simeq (E^\Phi,\sigma)$

\item The torus $T$ is:
\[ T_{E,\sigma,\beta} (R) = \{ x \in (E\otimes_k R)^\times \mid  x\sigma(x) = 1,\; \text{if }\beta(x) = (x_E,x_{E^\Phi}) \text{ then } \Psi(x_{E^\Phi}) = x_E \}.  \]
\end{itemize}
The rational conjugacy classes of $T_{E,\sigma,\beta}$ in $\Spin_{(L,V,Q,\twistedcomp)}$ are classified by:
\begin{itemize}
\item The rational conjugacy class of the torus $T_{E,\sigma}$ in $\SO_{(V,Q)}$.
\item The rational conjugacy class of the torus $T_{E^\Phi,\sigma}$ in $\Spin_{(V,Q)}$.
Note that the chosen torus must have image through the natural map the the chosen torus in $\SO_{(V,Q)}$.
\end{itemize}
\end{cor}

\begin{rmk}
As with Theorem \ref{thm:bigone} we point out that the
criterion under which  $(E,\sigma)$ is associated to a torus in $\SO_{(D,\tau)}$ are discussed in Section \ref{subsec:tori-clasic}, specifically Theorem \ref{thm:invariants_upgraded} or in the case of global fields Theorem \ref{thm:the_result_upgraded}.
The rational conjugacy classes of the torus $T_{E,\sigma}$ in $\SO_{(D,\tau)}$ are characterized in Theorem \ref{thm:lambdafortoriorth}.
The rational conjugacy classes of tori $T_{E^\Phi,\sigma}$ in $\Spin_{(D,\tau)}$  covering a particular torus are characterized in Theorem \ref{thm:CONGTORISPINCOVERS}.
\end{rmk}

\section*{Acknowledgements}

I would like to thank Andrei Rapinchuk for his encouragement to complete this project.

{}\ifx\XMetaCompile\undefined

\providecommand{\MR}[1]{}
\providecommand{\bysame}{\leavevmode\hbox to3em{\hrulefill}\thinspace}
\providecommand{\MR}{\relax\ifhmode\unskip\space\fi MR }
\providecommand{\MRhref}[2]{  \href{http://www.ams.org/mathscinet-getitem?mr=#1}{#2}
}
\providecommand{\href}[2]{#2}

\end{document}
{}\fi